\documentclass[dvipsnames]{scrartcl}
\usepackage{geometry}
\usepackage{tikz,float,easyReview,csquotes}
\usepackage{subcaption}
\usepackage[shortlabels]{enumitem}
\usetikzlibrary{arrows,calc,intersections,patterns}
\usepackage{mdframed}

\usepackage{meineshortcuts_paper}

\newcommand{\evcf}{N}
\newcommand{\admf}{N}

\newcommand{\mrcF}{\Phi}
\newcommand{\sumcount}{s}
\newcommand{\1}{\indic}

\author{Max Kämper%
  \and  Christoph Schumacher%
  \and Fabian Schwarzenberger%
  \thanks{Fakult\"at f\"ur Informatik/Mathematik, HTW Dresden, 01069 Dresden, Germany}
  \and Ivan Veseli\'c%
  \thanks{Fakult\"at f\"ur Mathematik, TU Dortmund, 44221 Dortmund, Germany}
}
\begin{document}
\title{Quantitative concentration inequalities for the uniform approximation of the IDS}
\maketitle

\begin{abstract}
  The integrated density of states (IDS) is a fundamental spectral quantity for quantum Hamiltonians modeling condensed matter systems, describing how densely energy levels are distributed. 
  It can be interpreted as a volume-averaged spectral distribution. Hence, there are two equivalent definitions of the IDS related by the Pastur–Shubin formula: an operator-theoretic trace formula and a limit of normalized eigenvalue counting functions on finite volumes.
We study a discrete random Schrödinger operator with bounded random potentials of finite-range correlations and prove a quantitative concentration inequality ensuring, with explicit high probability, that the empirical IDS (normalized eigenvalue counting function) uniformly approximates the abstract IDS trace formula within a prescribed error, thereby implying confidence regions for the IDS.
\medskip


\end{abstract}

%\thanks{...}

%\subjclass{47B80, 60B12, 62E20, 82B10}
%\keywords{Statistical mechanics,Anderson model, Integrated density of states, Uniform convergence, Empirical measures, Large deviations}

\section{Introduction}
The integrated density (IDS) of states is a key spectral characteristic of quantum Hamiltonians $H_\omega$ modelling condensed matter. It is the volume-normalized or volume-averaged spectral distribution function of such Hamiltonians, where some form of translation invariance, at least ergodicity is assumed. Naively speaking it informs us where the spectrum is more dense and where less so. It can be used to calculate all basic thermodynamic quantities of the corresponding non-interacting many-particle system.

There are two a priori independent ways to define the IDS. One is a closed formula of operator algebraic flavour and the other a limit formula. The former one reads
\begin{align} \label{eq:IDS-algebra}
	\RR\ni E \mapsto \Tr[\1_{(-\infty,E]}(H_\omega) \1_{\Lambda}]
\end{align}
where $\Tr$ denotes the trace, $\1_{(-\infty,E]}(H_\omega)$ the spectral projector on the energy interval $(-\infty,E]$ and $\1_{\Lambda}$ the multiplication operator with the indicator function of the unit cube $\Lambda=[-1/2, 1/2)^d$. The second definition is
\begin{equation} \label{eq:IDS-limit}
	\RR\ni E \mapsto \lim_{L\to\infty} N_\omega^{\Lambda_L}(E)
\end{equation}
where $ N_\omega^{\Lambda_L}$ denotes the normalized eigenvalue counting function of $ H_\omega^{\Lambda_L}$, 
some appropriate restriction of the Hamiltonian to the box $\Lambda_L:=[-L/2, L/2)^d$. Clearly, the equality of these two expressions is a variant of the law of large numbers or ergodic theorem. The identity is associated with the names of L.A. Pastur and M.A. Shubin, who established this formula for somewhat different classes of operators in their groundbreaking work, see for instance \cite{Shubin79,Pastur80}, the monographs  \cite{CarmonaLacroix90, PasturFigotin92,Stollmann01,Veselic08,AizenmanWarzel15} and the references therein.

More recently there has been interest in establishing a strong version of the Pastur-Shubin formula where convergence holds uniformly in the energy parameter $E$, see for instance
\cite{Lenz-02,LenzS-05,LMV06,LenzV-09,LSV10,PogorzelskiS12,Schwarzenberger12,AyadiSV13,SSV17,SSV18}.
Since all arguments (known to us) used in this context rely on the interlacing property of finite rank perturbations we will restrict ourselves in the rest of the paper to Hamiltonians defined on $\ell^2$ of some translation invariant graph. (Most ideas have been extended to quantum graphs of the same form, in particular to one-dimensional continuum random Schrödinger operators.)
Obviously, uniform convergence w.r.t. energy is stronger than pointwise (a.e.) convergence. Moreover, uniform convergence of probability distribution functions induces a natural topology on the space of measures.

From the empirical point of view only the quantity \eqref{eq:IDS-limit} is accessible.
Since one aims to identify the underlying measure defined by \eqref{eq:IDS-algebra} from the limiting procedure the following questions are natural:
%The question we will ask to motivate our new results will take a practical or empirical point of view.
\begin{itemize}
  \item \emph{How precisely can one predict or estimate the IDS based on a sample of laboratory measurements or, more realistically, a number of computational physics simulations?
How many measurements or iterations are necessary to be able to do so with a predescibed accuracy and confidence?}
  \item \emph{In the language of statistics we are asking for a confidence region of the IDS. An alternative scenario would be that a certain IDS is hypothesized for a material or Hamiltonian and a test is performed based on a number of samples. Then the  question is, how many samples are necessary, to distingush the hypothesized  measure?}
\end{itemize}

To give an idea how our results can answer such questions we formulate a consequence of the theorems obtained in the paper for a particulary simple model: Let $H_\omega=-\Delta +V_\omega\colon \ell^2(\ZZ^d) \to \ell^2(\ZZ^d) $, where $\Delta$ is the discrete Laplacian and $ (V_\omega \varphi)(k)= \omega_k \varphi(k)$ the multiplication operator by a sequence of i.i.d. bounded random variables $\omega_k$ indexed by $k \in \ZZ^d$. Denote the distribution of $\o_k$ by $\mu$ and $\PP=\bigotimes_{\ZZ^d} \mu$. Let $H_\omega^{\Lambda_L}=p_{\L_L}H_\o i_{\L_L}$ be the compression to the subspace $\ell^2(\Lambda_L)$ and $ N_\omega^{\Lambda_L}$ its normalized eigenvalue counting function.
Our results imply the following confidence region estimate:
\begin{tthm} \label{thm:reducedmainresult} For $d \geq 3$ and all $\alpha, \beta \in (0,1)$
	\begin{align*}
	\sup\limits_{\mu} \PP\{ \Vert N_\omega^{\Lambda_L}-N\Vert_\infty>\beta \} < \alpha
	\end{align*}
	provided
	\begin{align}\label{eq:minsizeforconcentrationineq}
		L>\max \left\{ \left(\frac{C}{\beta}+1\right)^2, \left(\left(\log \left(2 / \alpha \right)K\right)^{2/(d-2)}+1 \right)^2, 16  \right\}
	\end{align}
	where
	\begin{align*}
		C=40d+104 \cdot 2^d-51  \text{ and } 
		K =\left(\frac{480}{\log (3/2)}+\frac{16}{\log (2)}\right)<1207.
	\end{align*}
	Here the supremum is taken over all Borel probability measures $\mu$ on $\RR$ with bounded support.
\end{tthm}
Hence the $\alpha-$confidence region corresponding to the measurement $\o \in \RR^{\ZZ^d}$ is
\begin{align*}
	\left\{ \rho \in \BB \mid \rho \text{ isotone }, \norm[\infty]{N_\o^{\L_L}-\rho}\leq \beta \right\},
\end{align*}
provided \eqref{eq:minsizeforconcentrationineq} is satisfied. Here $\BB$ is the space of bounded right-continuous functions.

\subsubsection*{The rest of the paper ist organised as follows:}
We present the framework of our general model in Section  \ref{s:NAbP} and our main results in Section  \ref{s:MR},
concluding with a discussion of the results and an outline of the proof.
The latter is split into two parts:
The geometric one is presented in Section  \ref{sec:geometric} and the probabilistic in Section  \ref{sec:Probability} .
The proofs of the main theorems are completed in Section  \ref{sec:proofsofthemaintheorems}, whereas
Section \ref{s:EoS} discuses some extensions, in particular to random Schr\"odinger operators on finitely generated amenable groups 
and Laplace operators on long-range percolation graphs.

More general results and detailed proofs are available in the dissertation \cite{Diss}, which is freely accessible online. Accordingly, we present here some arguments only under simplfying conditions rather than in full detail, and rather focus on explaining how the statements for the more specific situation considered in this paper follow from the general framework developed in the thesis.

\section{Notation,  model assumptions, and some basic properties}
\label{s:NAbP}

Here we fix some notation, state assumptions \ref{prop:M1_translationinvariance} and \ref{prop:M3_independence} on the probability space under consideration, 
and thereafter establish fundamental properties of the eigenvalue counting functions, denoted here \ref{prop:admissibletranslationinvariance} to \ref{prop:measurable}.

Let $\norm[1]{x}=\sum_{i=1}^d \abs{x_i}$ denote the $\ell^1$-norm in $\ZZ^d$ with the associated metric
\begin{align*}
	\d_{\ZZ^d} \colon \ZZ^d \times \ZZ^d \to \NN_0, \ \d_{\ZZ^d} (x,y)=\norm[1]{x-y}.
\end{align*}
For $\L_1, \L_2 \subseteq \ZZ^d$ we define the distance between $\L_1$ and $\L_2$ as
\begin{align*}
	\d_{\text{set}}(\L_1, \L_2):=\min \left\{ \d_{\ZZ^d} (x,y) \mid x \in \L_1, y \in \L_2   \right\}.
\end{align*}
For any $\L \subset \ZZ^d$ we denote the number of points in $\L$ by $\abs{\L}$.\\
Let $\O=(\RR^{\ZZ^d},\cB (\RR^{\ZZ^d}), \PP)$, where $\cB (\RR^{\ZZ^d})$ is the Borel $\sigma$-algebra of $\RR^{\ZZ^d}$ and  $\PP$ a probability measure such that the following hold:
\begin{enumerate}[label={(M\arabic*)}]
	\item \textbf{Translation invariance}: The translations
	\begin{align}\label{eq:deftranslationomega}
	\gamma_z : \O \to \O, (\gamma_z \o)_y := \o_{y+z}.
	\end{align}	
	are measurable and $\PP \circ \gamma_z = \PP$ for each $z \in \ZZ^d$. \label{prop:M1_translationinvariance}
	%\item \textbf{Measurability of Projections}: For all $\L \in \cC$ the projection $\Pi_\L$ is measurable. \alert{kann man weglassen?} \label{prop:M2_existencedensities}
	\item \textbf{Independence at a distance}: There is an $r \geq 0$ such that for all $k \in \NN$ if
	\begin{itemize}
		\item $\L(1), ..., \L(k)$ are finite subsets of $\ZZ^d$,
		\item $\abs{\L(i)}\neq 0$ for all $1\leq i \leq k$ and
		\item $\min \{ \d_{\text{set}}(\L(i),\L(j)) \mid i \neq j \}>r$
	\end{itemize}
	then the projections $\left(\Pi_{\L(i)}\right)_{1\leq i\leq k}$
	\begin{align*}
		\Pi_\L : \O \to \O_\L , \ \Pi_\L (\o):=\o_\L := (\o_z)_{z \in \L}.
	\end{align*}	
	are independent, where $\O_\L=(\RR^{\L},\cB (\RR^{\L}), \PP_\L)$ with the image measure $\PP_\L= \PP\circ\Pi_\L^{-1}$. %\alert{so ist die Definition genau genommen zirkulär}
	\label{prop:M3_independence}
\end{enumerate}
%Let $\left(v_\cdot (z)\right)_{z \in \ZZ^d}$ be a sequence of identically distributed real random variables on $\O$.\\
The operator $H_\o = -\Delta + V_\o :  \ell^2 (\ZZ^d) \to \ell^2 (\ZZ^d)$ with
\begin{align}\label{eq:defAndersonop}
	H_\o \Psi (x) := \sum\limits_{y \in \ZZ^d : \ \abs{y-x}=1} \left(\Psi (x)- \Psi (y) \right) + \o_x \Psi (x)
\end{align}
for $\o \in \O$ is called the \textbf{Anderson operator}.\\
The Anderson operator can be compressed to a finite subset $\L$ of $\ZZ^d$, resulting in the \textbf{restricted Anderson operator}
\begin{align*}
	H_\o^\L : \ell^2 (\L) \to \ell^2 (\L), \ H_\o^\L = p_\L H_\o i_\L.
\end{align*}
Here, $i_\L: \ell^2 (\L) \to \ell^2 (\ZZ^d)$ and $p_\L : \ell^2 (\ZZ^d) \to \ell^2 (\L)$ are the \textbf{embedding} of $\ell^2 (\L)$ into $\ell^2 (\ZZ^d)$ and the \textbf{projection} from $\ell^2 (\ZZ^d)$ to $\ell^2 (\L)$ with
\begin{align*}
	i_\L \phi (z)=\begin{cases}
						\phi (z) & z \in \L \\
						0        & z \notin \L
				  \end{cases}, \quad p_\L \phii (z)= \phii (z) \ \forall z \in \L.
\end{align*}		
The restricted Anderson operator is a hermitian matrix operator, since $H_\o$ is self-adjoint and $\ell^2 (\L)$ is finite dimensional.
Thus, the operator has a finite number of real eigenvalues. The functions
\begin{align}
	\evcf (\L, \o) \colon \RR \to  [0,1], \ &\evcf (\L, \o) (x) :=  \#\left\{\text{eigenvalues of $H_\omega^\Lambda$ $\leq x$}\right\} \label{eq:defevcf} \\
	\bar{N}(\L, \o) \colon \RR \to  [0,1], \ &\bar{N}(\L, \o) (x) :=  \frac{\evcf (\L, \o)}{\abs{\Lambda}}. \label{eq:defnevcf}
\end{align}
are called the \textbf{eigenvalue counting function (evcf)} and the normalized  eigenvalue counting function 
for a restricted Anderson operator $H_\omega^\Lambda$, respectively.
All eigenvalues are always counted with their multiplicity.
We want to investigate the convergence of this function for larger and larger sets $\L$.

In order to discuss certain useful properties of the evcf, we first introduce some notation. Let $\cC$ be the set of all finite subsets of $\ZZ^d$.
We call
\begin{align*}
	\partial^r (\L) := \{x \in \L : \d_{\text{set}} (x, \ZZ^d \setminus \L) \leq r\} \cup \{x \in \ZZ^d \setminus \L : \d_{\text{set}} (x, \L) \leq r\}
\end{align*}
the \textbf{$r$-boundary} of a set $\L \subset \ZZ^d$ and define
\begin{align*}
	\L^r := \L \setminus \partial^r (\L)=\{x \in \L \mid \d_{\text{set}} (x, \ZZ^d \setminus \L) > r\}.
\end{align*}
For sets $\L_i$ with an index we use the shorthand $\L_i^r=\left(\L_i\right)^r$.
\\
We call any integer translate of $\{x \in \ZZ^d : \ 0\leq x_i <n \ \forall 1 \leq i \leq d\}$ a \textbf{cube of side length $n$}, which always contains $n^d$ elements.
In the following we will heavily use the fact that every sequence $\L_k, k \in \NN,$ of cubes of strictly increasing side length is a so called \textbf{F{\o}lner sequence}, meaning
\begin{align}\label{eq:Folnercubesbound}
	\frac{\abs{\partial^r \left(\L_k\right)}}{\abs{\L_k}}\xrightarrow[i \to \infty]{}0
\end{align}
for all $r \in \NN$.\\
Now we can formulate a number of properties of the eigenvalue counting function that we will use later.
\begin{enumerate}[label={(A\arabic*)}]
		\item \label{prop:admissibletranslationinvariance}\textbf{translation invariance}: For $\L \in \cC, \ z \in \ZZ^d$ and $\o \in \O$ we have
		\begin{align*}
			\evcf (\L+z, \o)=\evcf (\L, \gamma_z \o).
		\end{align*}
		\item \label{prop:admissiblelocality}\textbf{locality}: For all $\L \in \cC$ and $\o, \o' \in \O$ with $\Pi_\L (\o) = \Pi_\L (\o')$ we have
		\begin{align*}
			\evcf (\L, \o)=\evcf (\L, \o').
		\end{align*}
		\item \label{prop:admissiblealmostadditivity}\textbf{almost additivity}: For all $\o \in \O$, all pairwise disjoint $\L(1), ..., \L(n) \in \cC$ and $\L :=\bigcup_{i=1}^n \L(i)$ we have
		\begin{align*}
			\norm[\infty]{\evcf  (\L, \o) - \sum\limits_{i=1}^n \evcf  (\L(i), \o)} \leq \sum\limits_{i=1}^n b(\L(i)) \quad \text{ with }  b(\L):=8 \abs{\L \setminus \L^1} 
		\end{align*}
		$b$ satisfies
		\begin{itemize}
			\item for all $\L \in \cC$ and $z \in \ZZ^d$ we have $b(\L)=b(\L +z)$
			\item $b(\L)\leq 8 \abs{\L}$ for all $\L \in \cC$
			\item if $(\L(n))_{n \in \NN}$ is a sequence of cubes with strictly increasing side length, then
				\begin{align*}
					\lim\limits_{n \to \infty} \frac{b(\L(n))}{\abs{\L(n)}}=0
				\end{align*}
		\end{itemize}
		\item \label{prop:admissibleboundedness}\textbf{boundedness}:
		\begin{align*}
			\sup\limits_{\o \in \O} \norm[\infty]{\evcf (\{0\}, \o)}=1
		\end{align*}
		\item \label{prop:admissiblemonotone}\textbf{monotonicity}: The function $\evcf (\L, \o)$ is monotone increasing, i.e.
		\begin{align*}
			\forall \L \in \cC, \o \in \O: \ x<y \Longrightarrow \evcf (\L, \o) (x) \leq \evcf (\L, \o) (y).
		\end{align*}
		\item \label{prop:measurable}\textbf{point-wise measurability}: The function $\evcf (\L, \cdot) (x): \left(\RR^{\ZZ^d},\cB \left(\RR^{\ZZ^d}\right)\right) \to (\RR, \cB(\RR))$ is measurable for all $x \in \RR$ and $\L \in \cC$, where $\cB (\RR)$ is the Borel $\sigma$-algebra of $\RR$. %\alert{ist das notwendig?}
\end{enumerate}
\begin{proof}
	For the proof of \ref{prop:admissibletranslationinvariance} to \ref{prop:admissiblemonotone} see \cite[Lemma 7.1]{SSV17}.\\
	To show \ref{prop:measurable}, let $\l_j^\downarrow (H_\o^\L)$ be the j-th largest eigenvalue of $H_\o^\L$, counted with multiplicities. We use Corollary III.2.6 of \cite{Bhatia96}, which shows that for all $\L \in \cC$ and $\o \neq \o' \in \O$ we have
	\begin{align*}
		\max\limits_{j=1,...,\abs{\L}}\abs{\l_j^\downarrow (H_\o^\L) - \l_j^\downarrow (H_{\o'}^\L)}\leq \norm[]{H_\o^\L - H_{\o'}^\L}
	\end{align*}
	where the last norm is the operator norm. From the definition of the restricted Anderson operator follows
	\begin{align*}
		\norm[]{H_\o^\L - H_{\o'}^\L}=\norm[]{p_\L (V_\o - V_{\o'}) i_\L}=\norm[]{p_\L (V_{\o-\o'}) i_\L}=\norm[\infty]{\o_\L-{\o'}_\L}.
	\end{align*}
	%\alert{Passt das so?}
	Thus, $\l_j^\downarrow (H_\cdot^\L): \O^\L \to \RR$ is continuous and therefore measurable for all $j$.
	This in turn means that %$\{\o \in \O : \l_j^\downarrow (H_\o^\L)>x\}$ and also \\
	$\{\o \in \O : \l_j^\downarrow (H_\o^\L)\leq x\}$ is a measurable set for all $x \in \RR$ and all $1\leq j \leq \abs{\L}$, and consequently the functions $\indic \left\{\l_j^\downarrow (H_\o^\L)\leq x\right\}$ are also measurable. Finally
	\begin{align*}
		\evcf (\L, \o)(x)=\sum\limits_{j=1}^{\abs{\L}} \indic \left\{\l_j^\downarrow (H_\o^\L)\leq x\right\},
	\end{align*}
	which shows that $\o \mapsto\evcf (\L, \o)(x)$ is measurable for all $x \in \RR$ and $\L \in \cC$.
\end{proof}

\section{Main results: Concentration inequalites for the evcf}
\label{s:MR}

With the notation in place we are now able to state our main results. First is a concentration inequality for normalized eigenvalue counting functions. This result is an adaption of \cite[Corollary 7.3]{Diss}.

\begin{tthm}\label{thm:sqrtexpconcentrationinequality}
Let $d\geq 3$, $M \geq 2$
and $\O=(\RR^{\ZZ^d},\cB (\RR^{\ZZ^d}), \PP)$ be a probability space satisfying \ref{prop:M1_translationinvariance} and \ref{prop:M3_independence} with $r$ the smallest constant such that \ref{prop:M3_independence} is satisfied.
Let $N$ be the IDS and $\bar{N}(\L_n, \o)$ be the normalized eigenvalue counting function of an Anderson operator as in \eqref{eq:defAndersonop}
for a cube $\L_n \subset \ZZ^d$  of side length  $n \in \NN$.
Then there is a set $A_{M,n} \in \cB(\RR^{\ZZ^d})$ such that
	\begin{align}\label{eq:evcfbound}
		\norm[\infty]{\bar{N}(\L_n, \o) -N} &\leq 32 d \frac{1}{n} +104 \left(2^d-1\right) \frac{1}{\sqrt{n}} + \left( 8d + 4r (2^d-1) + 72dr +1 \right) \frac{1}{\sqrt{n}-1}
	\end{align}
	for all $\o \in A_{M,n}$ and
	\begin{align}\label{eq:evcfconcentration}
		\PP \left( A_{M,n} \right) \geq
		1-M \exp \left(- \frac{\sqrt{\lfloor n/\lfloor\sqrt{n}\rfloor \rfloor^d}}{\lfloor\sqrt{n}\rfloor K_M}\right)
	\end{align}
	provided $n>(2r+1)^2$ and $n>16$, where
	\begin{align*}
		K_M&=\left(\frac{40(M+1)}{\log (3/2) (M-1)}+\frac{4}{\log (M)}\right)\sum\limits_{q=0}^\infty 2^{-q} \sqrt{2+ 2 q \log (2)} \\
		&< 16 \left(\frac{10(M+1)}{\log (3/2) (M-1)}+\frac{1}{\log (M)}\right).
	\end{align*}
\end{tthm}

\begin{remark}
	Theorem \ref{thm:reducedmainresult} is a direct consequence of Theorem \ref{thm:sqrtexpconcentrationinequality}. First we choose $M=2$ and $r=0$, then note that \eqref{eq:evcfbound} implies
	\begin{align*}
		\norm[\infty]{\bar{N}(\L_n, \o) -N} < \frac{C}{\sqrt{n}-1}
	\end{align*}
	and \eqref{eq:evcfconcentration} implies
	\begin{align*}
		\PP \left( A_{2,n} \right) \geq 1 - 2 \exp \left(- \frac{\left(\sqrt{n}-1\right)^{d/2-1}}{K}\right)
	\end{align*}
	with $C$ and $K$ as in the statement of Theorem \ref{thm:reducedmainresult}. Then we just have to choose $n$ large enough to ensure
	\begin{align*}
		\frac{C}{\sqrt{n}-1} \leq \beta , \ \PP \left( A_{2,n} \right) \geq 1 - \alpha
	\end{align*}
	as well as the conditions $n>(2r+1)^2$ and $n>16$ necessary for Theorem \ref{thm:sqrtexpconcentrationinequality}, leading to the bound given in Theorem \ref{thm:reducedmainresult}.
\end{remark}

\begin{remark}[What is different for dimensions one and two?]\label{cor:sqrtexpconcentrationinequality1+2dim}
	For dimensions below three the concentration inequality needs to have a slightly different scaling in $n$. 
The corresponding result to Theorem \ref{thm:sqrtexpconcentrationinequality} for $d=1$ and $d=2$ provides an error bound differing from \eqref{eq:evcfbound} only in the last term, 
namely 
	\begin{align}
	\label{eq:evcfbound:k}
		\norm[\infty]{\bar{N}(\L_n, \o) -\admf} &\leq 32 d \frac{1}{n} +104 \left(2^d-1\right) \frac{1}{n^{1-1/k}} +   \frac{8d + 4r (2^d-1) + 72dr +1 }{\sqrt[k]{n}-1}
		%+ \left( 8d + 4r (2^d-1) + 72dr +1 \right) \frac{1}{\sqrt[k]{n}-1}
	\end{align}
	for all $\o \in A_{M,n}'$ where now the event $A_{M,n}'$ satisfies
	\begin{align*}
		\PP \left( A_{M,n}' \right) \geq
		1-M \exp \left(- \frac{\sqrt{\lfloor n/\lfloor\sqrt[k]{n}\rfloor \rfloor^d}}{\lfloor\sqrt[k]{n}\rfloor K_M}\right)
	\end{align*}
    where
    \begin{align*}
		k=\begin{cases}
			4 & \text{ for }d=1\\
			3 & \text{ for }d=2\\
		\end{cases}
	\end{align*}
	but the other constants are unchanged. For a proof see \cite[Corollary 7.3]{Diss}.
\end{remark}

With some changes of the probabilistic parts of the proofs of the previous results it is also possible to show a concentration inequality where the exponent is improved to $\frac{\lfloor n/\lfloor\sqrt[k]{n}\rfloor \rfloor^d}{\lfloor\sqrt[k]{n}\rfloor}$ (instead of $ \frac{\sqrt{\lfloor n/\lfloor\sqrt[k]{n}\rfloor \rfloor^d}}{\lfloor\sqrt[k]{n}\rfloor}$) under a stronger assumption on the size of $n$. 
%We only state the result here, the full proof can be found in \cite[Corollary 7.5]{Diss}.

\begin{tthm}\label{thm:subexpconcentrationinequality}
	Let $ d\in \NN$ as well as $k =2$ for $d\geq 5$ and $k >\frac{4+d}{d}$ for $d<5$. 
%	\begin{align*}
%		k\begin{cases}
%			>\frac{4+d}{d} & \text{ for } d<5\\
%			=2             & \text{ for } d\geq 5
%		\end{cases},
%	\end{align*}
	In the setting of Theorem \ref{thm:sqrtexpconcentrationinequality} there is a set $B_n \in \cB(\RR^{\ZZ^d})$ such that \eqref{eq:evcfbound:k} is true for all $\o \in B_n$ and
	\begin{align*}
		\PP\left(B_n \right)
		\geq 1-  \exp \left( -\frac{1}{24}\frac{\lfloor n/\lfloor\sqrt[k]{n}\rfloor \rfloor^d}{\lfloor\sqrt[k]{n}\rfloor} \right)
	\end{align*}
	provided %$n>(2r+1)^k$ and $n>4$, and additionally
	%	\begin{align*}
		%		m(n)=\lfloor n^k \rfloor
		%	\end{align*}
	%	with $k<d/(4+d)$ for $d<5$ and $k=1/2$ for $ d \geq 5$, provided
	\begin{align*}
		n > \max \left\{16,(2r+1)^k, \left(\left(12 K_2\right)^{2/d}+1\right)^{\frac{dk}{dk-d-4}}\right\}
	\end{align*}
	where $K_2=\left(\frac{120}{\log (3/2)}+\frac{4}{\log (2)}\right) \sum\limits_{q=0}^\infty 2^{-q} \sqrt{2+ 2 q \log (2)}<\left(\frac{480}{\log (3/2)}+\frac{16}{\log (2)}\right)$.
\end{tthm}

%\begin{cor}[Special case r=0, M=2]
	
%\end{cor}

\subsection{Discussion of achievements and limitations of our new findings}
Our results establish that the abstract spectral distribution function $\RR\ni E \mapsto \Tr[\1_{(-\infty,E]}(H_\omega) \1_{\Lambda}]$ allows for explicit confidence regions based on sampling data. 
This is to our knowledge the first rigorous result for estimating the IDS in the \textbf{space of measures} in the language of statistics.
\\

 Our strong error bound result from to the use of concentration inequalities based on entropy and bracketing numbers, 
 together with the observation that due of the structure of the eigenvalue counting functions, bracketing covers inherit favorable properties, 
 leading in turn to efficient complexity bounds. The construction of bracket coverings for eigenvalue counting functions are to the best of our knowledge 
an entirely new contribution of this paper, respectively the Dissertation  \cite{Diss}.
\\

Our result certainly does not provide sharp bounds on the approximation error, or the number of necessary samples, respectively. The bounds are too large to be relevant for computational physicists.
This is on one hand due to the fact that we gave away much when estimating the constants in the proof. Even more importantly, the geometric error is at least of the order of magnitude $n^{-1/2}$, requiring a huge size of the individual samples.
This geometric error of order $n^{-1/2}$ is the bottleneck for any improvement of our estimates. Possibly an improvement could be achieved by some kind of plane-wave or quasi-momentum disintegration of the states in the $\ell^2$ space and averaging over quasi-periodic boundary conditions in the vein of Klopp's approach in  \cite{Klopp-99}.

Furthermore, the Hamiltonian we consider is already an approximation, since it is a one-electron Schrödinger operator, and is considered in the discretized version.
In computational physics one would turn at an earlier stage to an dimension-reduced effective model and determine confidence regions for such effective models. Supremum norms may not be the tool of choice there.

\subsection{Strategy of the proofs and outline for the remainder of the paper}
Our general strategy follows \cite{SSV18}, which is in turn inspired by earlier works, e.g.~\cite{LMV06} and \cite{LSV10}.
The techniques can be divided into two groups. 
The first is the subject of Section \ref{sec:geometric} and is based on almost-additivity and the geometry of $\ZZ^d$. 
The second is devoted to a quantification of the law of large numbers based on concentration inequalities and is covered in Section \ref{sec:Probability}.
Here we use a different approach than in the earlier mentioned works \cite{LMV06}, \cite{LSV10}, and \cite{SSV18}. \\
The proofs of the final concentration inequalities in turn consist of three steps:\\
We first approximate $\frac{\admf (\L_n, \o)}{\abs{\L_n}}$ by a normalized sum over a large number of independent, identically distributed random variables in Lemma \ref{lem:geometricboundevcf} of Section \ref{sec:geometric}.
Then we introduce the Orlicz norm in Section \ref{sec:Probability} and use it to outline a probabilistic argument (Theorem \ref{theorem:expofcountablesup}) to show that this averaged sum converges almost surely uniformly to an expected value for $n \to \infty$. This is accomplished via bracket coverings, which are the topic of Subsection \ref{subsec:bracketcoverings}. The argument yields explicit error bounds in the form of two concentration inequalities, Corollaries \ref{cor:firstconcineq} and \ref{cor:secondconcineq}.\\
After that we use once more geometric arguments (namely Theorem \ref{thm:normalizedadmissiblefctconvergence} of Section \ref{sec:geometric}) 
to show that the sequence of expected values (associated to the approximation of  $\frac{\admf (\L_n, \o)}{\abs{\L_n}}$) converges as well. \\
In Section \ref{sec:proofsofthemaintheorems} all the provious steps are assembled to yield the main results.\\

Symbolically, the strategy can be summarised as follows:
\begin{align}\label{eq:strategy}
	\frac{\admf (\L_n, \o)}{\abs{\L_n}} \approx \frac{1}{k(n)} \sum\limits_{i=1}^{k(n)} \frac{\admf (\L_{m,i}, \o)}{\abs{\L_{m,i}}}\xrightarrow[a.s.]{n\to \infty} \EE \left( \frac{\admf (\L_m, \o)}{\abs{\L_m}}\right) \xrightarrow[]{m \to \infty} \admf.
\end{align}
Here, $\L_n$ is a cube with side length $n$, $k(n)<n$ is a number that grows monotonously with $n$, $\left(\L_{m,i}\right)_{1\leq i \leq k(n)}$ are a series of cubes with side length $m<2n$ that do not intersect.

\section{Geometric approximation estimates}\label{sec:geometric}
We start with the geometric approximation argument. Let $\L_n := ([0,n) \cap \ZZ)^d$ be a cube with side length $n\in \NN$. We \textquote{tile} this cube by smaller cubes with side length $m$, such that $2m<n$ and define the \textbf{tiling set}
\begin{align*}
	T_{m,n}:=\{t \in T_m : \L_m +t \subset \L_n\}
\end{align*}
where $T_m:=m\ZZ^d = \{(z^{(i)})_{1\leq i \leq d} \in \ZZ^d : z^{(i)} \ \mathrm{ mod } \ m =0 \ \forall 1\leq i \leq d\}$. The cardinality of this set is $\abs{T_{m,n}}=\lfloor n/m \rfloor^d$ and all the sets $\L_m + t$ with $t \in T_{m}$ are pairwise disjoint. The union
\begin{align*}
	\L_{m,n}:=\bigcup_{t \in T_{m,n}} \left( \L_m + t\right)= \L_m + T_{m,n}.
\end{align*}
satisfies $\L_{m,n}=\L_{ \lfloor n/m \rfloor m }$. We further define
\begin{align*}
	\hat{\L}_{m,n}:=\L_n \setminus \L_{m,n}.
\end{align*}
See Figure \ref{fig:cubes} for an illustration.
	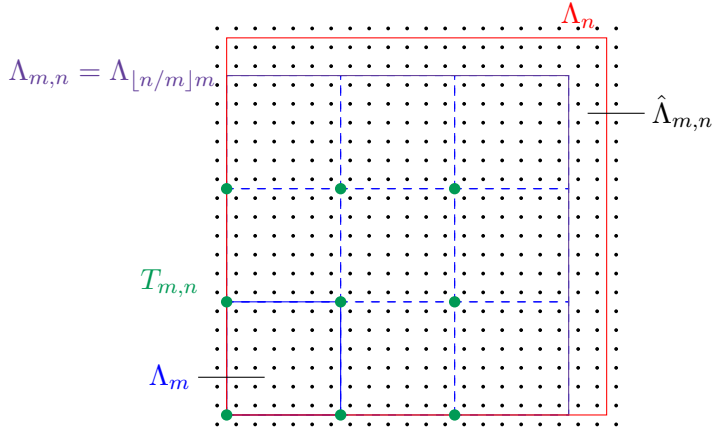
\begin{figure}[H]
	\centering
	\begin{tikzpicture}
		\begin{scope}[scale=0.25]
			\def\X{22};
			\def\Y{22};
			\foreach \x in {1,...,\X}{
				\foreach \y in {1,...,\Y}{
					\fill (\x,\y) circle (0.1);
				}
			};
			\begin{scope}[shift={(1.5,1.5)}]
				\foreach \c in {0,...,2}{
					\foreach \d in {0,...,2}{
						\draw[blue, dashed] (6*\c,6*\d) rectangle (6*\c+6,6*\d+6);
					}
				};
			\draw[blue] (0,0) rectangle (6,6);
			\end{scope}
			\filldraw[draw=white,fill=white] (0,2) rectangle node[blue]{$\Lambda_m$} (-3,5);
	%\node[blue] at (12.5,12.5) {$\Lambda_m$};

			\draw[RoyalPurple] (19.5,1.5) rectangle (1.5,19.5) node [left] {$\L_{m,n}=\L_{ \lfloor n/m \rfloor m   }$};			
			\draw[red] (1.5,1.5) rectangle (21.5,21.5) node [above left] {$\Lambda_n$};
			\foreach \a in {0,6,12}{
				\foreach \b in {0,6,12}{
					\fill[color=ForestGreen]  (\a+1.5,\b+1.5) circle (0.3);
				}
			};
			\filldraw[draw=white,fill=white] (0,7) rectangle node[ForestGreen]{$T_{m,n}$} (-3,10);
			\draw (0,3.5)--(3.5,3.5);
			\draw (20.5,17.5)--(23.5,17.5);
			\filldraw[draw=white,fill=white] (25,17) rectangle node[black]{$\hat{\L}_{m,n}$} (26,18);			
		\end{scope}	
	\end{tikzpicture}
	\caption{Illustration of some lattice subsets for $d=2$.}
	\label{fig:cubes}
	\end{figure}
The properties of the eigenvalue counting function ensure that for all $\o \in \O$, and all $n,m,r \in \NN$ with $n>2m$, $m > 2r+1$
%\begin{align}\label{eq:geometricbounds}
%		\norm[\infty]{\frac{\evcf  (\L_n, \o)}{\abs{\L_n}} -\frac{1}{\abs{T_{m,n}}} \sum\limits_{t \in T_{m,n}} \frac{\evcf  (\L^r_m + t , \o) }{\abs{\L_m}} }
%		&\leq \frac{b\left( \L_{ \lfloor n/m \rfloor m } \right)}{\abs{\L_{ \lfloor n/m \rfloor m }}}+\frac{10 \abs{\partial^m \left(\L^m_n \right)} }{\abs{\L_n^m}}\nonumber \\
%		&\quad  +\frac{b(\L_m) + b(\L_m^r)+ 9\abs{\partial^r \left(\L_m \right)}}{\abs{\L_m}}.
%		%&\quad \quad \leq \frac{b\left( \L_{ \lfloor n/m \rfloor m } \right)}{\abs{\L_{ \lfloor n/m \rfloor m }}}+\frac{(3 D + 2 E) \abs{\partial^m (\L_n^m)} }{\abs{\L_n^m}}+\frac{b(\L_m) + b(\L_m^r)+ (2D+E)\abs{\partial^r (\L_m)}}{\abs{\L_m}}.
%\end{align}
\begin{align}\label{eq:geometricbounds}
	\norm[\infty]{\frac{\evcf  (\L_n, \o)}{\abs{\L_n}} -\frac{1}{\abs{T_{m,n}}} \sum\limits_{t \in T_{m,n}} \frac{\evcf  (\L^r_m + t , \o) }{\abs{\L_m}} }
	&\leq \frac{b\left( \L_{ \lfloor n/m \rfloor m } \right)}{\abs{\L_{ \lfloor n/m \rfloor m }}}+\frac{26 \left(\abs{\L_n}-\abs{\L_n^m}\right) }{\abs{\L_n^m}}\nonumber \\
	&\quad  +\frac{b(\L_m) + b(\L_m^r)+ 17\left(\abs{\L_m}-\abs{\L_m^r}\right)}{\abs{\L_m^r}}. %\frac{b(\L_m) + b(\L_m^r)+ 9\abs{\partial^r \left(\L_m \right)}}{\abs{\L_m}}.
	%&\quad \quad \leq \frac{b\left( \L_{ \lfloor n/m \rfloor m } \right)}{\abs{\L_{ \lfloor n/m \rfloor m }}}+\frac{(3 D + 2 E) \abs{\partial^m (\L_n^m)} }{\abs{\L_n^m}}+\frac{b(\L_m) + b(\L_m^r)+ (2D+E)\abs{\partial^r (\L_m)}}{\abs{\L_m}}.
\end{align}
For an explicit calculation see \cite[Lemma 5.1]{Diss}. % \cite[Lemma 4.1]{SSV17}.
Next we want to make this expression more explicit in terms of $n$,$m$,$r$ and $d$.
\begin{llem}\label{lem:geometricboundevcf}
	For eigenvalue counting functions  on $d$-dimensional cubes with $n,m \in \NN$, $n>4m$ and $m > 2r+1$ we have
	\begin{align}\label{eq:geometricexplicitbound}
		\norm[\infty]{\frac{\evcf  (\L_n, \o)}{\abs{\L_n}} -\frac{1}{\abs{T_{m,n}}} \sum\limits_{t \in T_{m,n}} \frac{\evcf  (\L^r_m + t , \o) }{\abs{\L_m}} }  &\leq 32 d \frac{1}{n} + 104 \left(2^d-1\right) \frac{m}{n} \\
		&\quad + \left( 4d + 2r (2^d-1) + 36dr \right) \frac{1}{m} \nonumber
	\end{align}
\end{llem}
\begin{proof}
The boundary function $b$ is $b(\L)=8\abs{\L \setminus \L^1}$ by \ref{prop:admissiblealmostadditivity} and additionally we have $\abs{\L_n}=n^d$ and $\abs{\L_n^r}=(n-2r)^d$ for cubes $\L_n$. Thus
\begin{align*}
	\frac{b\left( \L_{ \lfloor n/m \rfloor m } \right) }{\abs{\L_{ \lfloor n/m \rfloor m }}}&=8\frac{\abs{\L_{ \lfloor n/m \rfloor m }} - \abs{\L_{ \lfloor n/m \rfloor m }^1}}{\abs{\L_{ \lfloor n/m \rfloor m }}}=8\left(1- \left( \frac{\lfloor n/m \rfloor m -2}{\lfloor n/m \rfloor m } \right)^d\right).
	%\abs{\L_n}-\abs{\L_n^m} &= n^d - (n-2m)^d \\
	%\abs{\L_m}-\abs{\L_m^r} &= m^d - (m-2r)^d.
\end{align*}
and consequently
\begin{align}\label{eq:evcfconcentrationroughbound}
	&\norm[\infty]{\frac{\evcf  (\L_n, \o)}{\abs{\L_n}} -\frac{1}{\abs{T_{m,n}}} \sum\limits_{t \in T_{m,n}} \frac{\evcf  (\L^r_m + t , \o) }{\abs{\L_m}} } \\
	&\quad\quad\quad\leq 8\left(1- \left( \frac{\lfloor n/m \rfloor m -2}{\lfloor n/m \rfloor m } \right)^d\right) + 26 \left( \left(\frac{n}{n-2m}\right)^d - 1 \right) \nonumber \\
	&\quad\quad\quad \quad+ \frac{m^d - (m-2)^d + (m-2r)^d - (m-2r-2)^d+  17 (m^d - (m-2r)^d)}{m^d} \nonumber\\
	&\quad\quad\quad=8\left(1- \left(1- \frac{2}{\lfloor n/m \rfloor m } \right)^d\right) + 26 \left( \left(1+\frac{2 m/n}{1-2m/n}\right)^d - 1 \right) \nonumber \\
	&\quad\quad\quad\quad+ \left(1-\left(1-\frac{2}{m}\right)^d \right)+\left(\left(1-\frac{2r}{m}\right)^d - \left(1-\frac{2r+2}{m}\right)^d\right)\nonumber 
%\\
%	&\quad\quad\quad\quad
+  17 \left(1 - \left(1- \frac{2r}{m}\right)^d\right)
\end{align}
This expression only holds for $m > 2r+1$, since we exclude the edge case where $\left(\L_m^{r}\right)^1 = \emptyset$ but $m-2r-2<0$ in favour of a unified expression.\\
Now we use Bernoulli's inequality
\begin{align}\label{eq:bernoulliineq}
	1-(1-y)^d\leq y d
\end{align}
for $y\leq 1$ for the first, third and fifth term on the right of \eqref{eq:evcfconcentrationroughbound}. Since 
$\lfloor n/m \rfloor m \geq n-m$ we have
\begin{align*}
	\frac{1}{\lfloor n/m \rfloor m }\leq \frac{1}{n-m}\leq \frac{1}{n}\frac{1}{1-m/n}\leq \frac{2}{n}.
\end{align*}
if $n>2m$. \\
If $n\geq 4$ we can apply Bernoulli's inequality and get  the bounds
\begin{align*}
	\left(1- \left(1- \frac{2}{\lfloor n/m \rfloor m } \right)^d\right)\leq \frac{2d}{\lfloor n/m \rfloor m}\leq \frac{4d}{n},\\
	\left(1-\left(1-\frac{2}{m}\right)^d \right) \leq \frac{2d}{m}   \quad  \text{ and } \quad 
	\left(1 - \left(1- \frac{2r}{m}\right)^d\right) \leq \frac{2dr}{m}
\end{align*}
since $m > 2r+1$ and thus $m\geq 2$ as well as $m \geq 2r$ was already required. Next we note that by the binomial theorem
\begin{align}\label{eq:binomialtheorm}
	\left(1+x\right)^d=\sum\limits_{j=0}^d \binom{d}{j}x^j
\end{align}
we have
\begin{align*}
	\left(1+x\right)^d -1 \leq \abs{x} \sum\limits_{j=1}^d \binom{d}{j}\abs{x}^{j-1}\leq \abs{x} \sum\limits_{j=1}^d \binom{d}{j} = \abs{x} (2^d -1)
\end{align*}
for $\abs{x}\leq 1$. From
\begin{align*}
	\frac{z}{1-z}\leq 2 z \leq 1 \Leftrightarrow z \leq \frac{1}{2}
\end{align*}
follows
\begin{align*}
	\left( \left(1+\frac{2 m/n}{1-2m/n}\right)^d - 1 \right) \leq \frac{4m}{n}\left(2^d-1\right)
\end{align*}
as long as $4m<n$.\\
For the fourth term we use both \eqref{eq:bernoulliineq} and \eqref{eq:binomialtheorm} for
\begin{align*}
	\left(1-\frac{2r}{m}\right)^d-\left(1-\frac{2r+2}{m}\right)^d\leq \left(2^d-1\right)\frac{2r}{m}+\frac{d (2r+2)}{m}=\frac{2r 2^d +2d (r+1)-2r}{m}.
\end{align*}
In conclusion we get
\begin{align*}
	\norm[\infty]{\frac{\evcf  (\L_n, \o)}{\abs{\L_n}} -\frac{1}{\abs{T_{m,n}}} \sum\limits_{t \in T_{m,n}} \frac{\evcf  (\L^r_m + t , \o) }{\abs{\L_m}} }  &\leq 32 d \frac{1}{n} + 104 \left(2^d-1\right) \frac{m}{n} \\
	&\quad + \left( 4d + 2r (2^d-1) + 36dr \right) \frac{1}{m} \nonumber
\end{align*}
by applying all the previous bounds to \eqref{eq:evcfconcentrationroughbound}.
\end{proof}
This result further yields the existence of a limit for the expected value of the normalized eigenvalue counting functions on growing cubes, namely
\begin{tthm}\label{thm:normalizedadmissiblefctconvergence}
	Let $\evcf$ be the eigenvalue counting function, $r,n \in \NN$ such that $2r+1<n$ and $\L_n := ([0,n) \cap \ZZ)^d$. Then $\frac{\EE \evcf  (\L_n^r, \cdot)}{\abs{\L_n}}$ (defined pointwise in $\RR$) forms a $\norm[\infty]{\cdot}$-Cauchy sequence and there exists a limit function $\admf^{*}\in \BB$ that is monotonically increasing. Furthermore
	\begin{align} \label{eq:normalizedadmissiblefctconvergence}
		\norm[\infty]{\frac{\EE \admf (\L_n^r, \cdot)}{\abs{\L_n}} - \admf^{*}} \leq \left( 4d + 2r (2^d-1) + 36dr \right) \frac{1}{n}.
	\end{align}
\end{tthm}

\begin{proof}
	First we note that because of \ref{prop:admissibletranslationinvariance} and the assumed property \ref{prop:M1_translationinvariance} of the probability space, the equation
	\begin{align*}
		\EE \admf (\L^r_m +t, \cdot)=\EE \admf (\L_m^r, \cdot) \circ \gamma_t =\EE \admf (\L_m^r, \cdot)
	\end{align*}
	is true for all $t \in \ZZ^d$ and $m,r \in \NN$ with $m>2r$.
	Therefore
	\begin{align}
		\EE \frac{1}{\abs{T_{m,n}}} \sum\limits_{t \in T_{m,n}} \frac{\admf (\L^r_m + t , \cdot) }{\abs{\L_m}}=  \frac{1}{\abs{T_{m,n}}} \sum\limits_{t \in T_{m,n}} \frac{\EE \admf (\L^r_m + t , \cdot) }{\abs{\L_m}} = \frac{\EE \admf (\L^r_m , \cdot) }{\abs{\L_m}}. \label{eq:expectedvaluenormalizedadmissiblefunction}
	\end{align}
	holds as well. We can thus reuse the bound of Lemma \ref{lem:geometricboundevcf} to bound
	\begin{align*}
		\norm[\infty]{\frac{\EE \admf (\L^r_n , \cdot) }{\abs{\L_n}} - \frac{\EE \admf (\L^r_k , \cdot) }{\abs{\L_k}}}
	\end{align*}
	as well. However, that bound only holds for $n>4m$ and we need to establish a Cauchy criterium for all $n,m$.\\
	For the next step, let $m<n \in \NN$ and $\delta>0$. Then by \ref{prop:admissiblealmostadditivity} and \eqref{eq:Folnercubesbound} there exists a $k>4n$ that is divisible by $n$ and $m$ (resulting in $\lfloor k/n\rfloor n=\lfloor k/m\rfloor m=k$) such that
	\begin{align*}
		\max \left\{ \frac{b(\L_k)}{\abs{\L_k}},\frac{\abs{\partial^n \left(\L^n_k \right)} }{\abs{\L_k^n}}, \frac{\abs{\partial^m \left(\L^m_k \right)} }{\abs{\L_k^m}} \right\} <\delta.
	\end{align*}
	We use the triangle inequality and \eqref{eq:expectedvaluenormalizedadmissiblefunction} to get
	\begin{align} \label{eq:expadmissiblefctCauchy1}
		&\norm[\infty]{\frac{\EE \admf (\L^r_n , \cdot) }{\abs{\L_n}} - \frac{\EE \admf (\L^r_m , \cdot) }{\abs{\L_m}}} \leq \norm[\infty]{\EE \frac{1}{\abs{T_{n,k}}} \sum\limits_{t \in T_{n,k}} \frac{\admf (\L^r_n + t , \cdot) }{\abs{\L_n}} - \frac{\EE \admf (\L_k , \cdot) }{\abs{\L_k}}}\nonumber \\
		&\quad \quad+ \norm[\infty]{\frac{\EE \admf (\L_k , \cdot) }{\abs{\L_k}} - \EE \frac{1}{\abs{T_{m,k}}} \sum\limits_{t \in T_{m,k}} \frac{\admf (\L^r_m + t , \cdot) }{\abs{\L_m}}}.
	\end{align}
	Both terms on the right side can be treated the same way by first considering a specific $x \in \RR$ and using \eqref{eq:geometricbounds} as well as the results of Lemma \ref{lem:geometricboundevcf} for
	\begin{align*}
		&\abs{\EE \frac{1}{\abs{T_{n,k}}} \sum\limits_{t \in T_{n,k}} \frac{\admf (\L^r_n + t , \cdot)(x) }{\abs{\L_n}} - \frac{\EE \admf (\L_k , \cdot) (x) }{\abs{\L_k}}} \\
		&\quad\leq \EE \abs{ \frac{1}{\abs{T_{n,k}}} \sum\limits_{t \in T_{n,k}} \frac{\admf (\L^r_n + t , \cdot) (x)}{\abs{\L_n}} -  \frac{\admf (\L_k , \cdot) (x) }{\abs{\L_k}}}\\
		&\quad  \leq \EE \left( \frac{b\left( \L_{k} \right)}{\abs{\L_{k}}}+\frac{26 \left(\abs{\L_k}-\abs{\L_k^n}\right) }{\abs{\L_k^n}}%\frac{10 \abs{\partial^n \left(\L^n_k \right)} }{\abs{\L_k^n}}
		+\frac{b(\L_n) + b(\L_n^r)+ 17\left(\abs{\L_n}-\abs{\L_n^r}\right)}{\abs{\L_n^r}} \right)\\
		%&\quad \quad \leq \EE \left( \frac{b(\L_k)}{\abs{\L_k}} + \frac{(3D+2E) \abs{\partial^n (\L_k^n)}}{\abs{\L_k^n}}+\frac{b(\L_n)+b(\L_n^r) +(2D+E)\abs{\partial^r (\L_n)}}{\abs{\L_n}} \right)\\
		&\quad  \leq 27 \delta +\frac{b(\L_n) + b(\L_n^r)+ 17\left(\abs{\L_n}-\abs{\L_n^r}\right)}{\abs{\L_n^r}}\\
		&\quad  \leq 27 \delta + \left( 4d + 2r (2^d-1) + 36dr \right) \frac{1}{n}.
	\end{align*}
	Since $x$ was arbitrary, this results in
	\begin{align}
		&\norm[\infty]{\EE \frac{1}{\abs{T_{n,k}}} \sum\limits_{t \in T_{n,k}} \frac{\admf (\L^r_n + t , \cdot) }{\abs{\L_n}} - \frac{\EE \admf (\L_k , \cdot) }{\abs{\L_k}}} 
%\nonumber \\		&\quad \quad
\leq 27 \delta +\left( 4d + 2r (2^d-1) + 36dr \right) \frac{1}{n} \label{eq:expadmissiblefctCauchy2}
	\end{align}
	where the second term on the right side converges to $0$ for increasing $n$. % as established in \eqref{eq:geometricboundconvergence}.
	Therefore, for every $\eps>0$ there is an $n_0\in \NN$ such that
	\begin{align}
		\left( 4d + 2r (2^d-1) + 36dr \right) \frac{1}{n} < \eps \ \forall n\geq n_0. \label{eq:expadmissiblefctCauchy3}
	\end{align}
	The bound \eqref{eq:expadmissiblefctCauchy2} is also valid if $n$ is replaced by  $m$. Thus, if $n,m \geq N$ then
	\begin{align*}
		\norm[\infty]{\frac{\EE \admf (\L^r_n , \cdot) }{\abs{\L_n}} - \frac{\EE \admf (\L^r_m , \cdot) }{\abs{\L_m}}} \leq 54 \delta + 2 \eps
	\end{align*}
	which follows from \eqref{eq:expadmissiblefctCauchy1}, \eqref{eq:expadmissiblefctCauchy2} and \eqref{eq:expadmissiblefctCauchy3}. This proves that $\frac{\EE \admf (\L^r_n , \cdot) }{\abs{\L_n}}$ is a Cauchy sequence and has a uniform limit $\admf^{*}$. Since $\frac{\admf (\L^r_n , \o) }{\abs{\L_n}}$ is right-continuous, monotone increasing and bounded by 1 for every $n$ and every $\o$ this is also true for its expected value by the bounded convergence theorem. %, see for example \cite[Corollary 6.26]{KlenkeENG}.
	Therefore, every $\EE \frac{\admf (\L^r_n , \cdot) }{\abs{\L_n}}$ lies in $\BB$ and is monotone increasing, and as a uniform limit this is true for $\admf^{*}$ as well.
	%	To see that $\admf^{*}$  is monotone increasing we use an argument from \cite{SSV17}: Let $\eps>0$ be arbitrary and choose $n \in \NN$ such that $\norm[\infty]{\EE \frac{\admf (\L^r_n , \cdot) }{\abs{\L_n}}-\admf^{*}}<\eps/2$. Then for all $x<x'\in \RR$ we have
	%	\begin{align*}
		%	\admf^{*}(x)\leq \EE \frac{\admf (\L^r_n , \cdot)(x) }{\abs{\L_n}}+\frac{\eps}{2}\leq \EE \frac{\admf (\L^r_n , \cdot)(x') }{\abs{\L_n}}+\frac{\eps}{2}\leq \admf^{*}(x)+\eps
		%	\end{align*}
	%	and thus $\admf^{*} (x) \leq \admf^{*}(x')$.\\
	By taking the limit $m \to \infty$ we can also obtain the bound \eqref{eq:normalizedadmissiblefctconvergence} for the difference of 
$\frac{\EE \admf (\L^r_n , \cdot) }{\abs{\L_n}}- \admf^{*}$ since $\delta$ was arbitrary.
\end{proof}

\section{Probabilistic approximation estimates} \label{sec:Probability}
Following the preceding section we can approximate the normalized eigenvalue counting function on a large cube by the averaged sum over the normalized eigenvalue counting functions on smaller cubes with explicit control on the error of the approximation. Next we would like to investigate the behaviour of this averaged sum. As an averaged sum of independent, identically distributed random variables we can hope to find a convergence to an expected value. In more detail, as a consequence of the two geometric approximation estimates \eqref{eq:geometricexplicitbound} and \eqref{eq:normalizedadmissiblefctconvergence} we have
\begin{align}\label{eq:completegeometricbound}
	\norm[\infty]{\frac{\evcf  (\L_n, \o)}{\abs{\L_n}} -\admf^{*}}&\leq \norm[\infty]{\frac{\evcf  (\L_n, \o)}{\abs{\L_n}} -\frac{1}{\abs{T_{m,n}}} \sum\limits_{t \in T_{m,n}} \frac{\evcf  (\L^r_m + t , \o) }{\abs{\L_m}} } \nonumber \\
	&\quad + \norm[\infty]{\frac{1}{\abs{T_{m,n}}} \sum\limits_{t \in T_{m,n}} \frac{\evcf  (\L^r_m + t , \o) }{\abs{\L_m}} - \frac{\EE \admf (\L_m^r, \cdot)}{\abs{\L_m}} } \nonumber \\
	&\quad + \norm[\infty]{\frac{\EE \admf (\L_m^r, \cdot)}{\abs{\L_m}} - \admf^{*}}\nonumber \\
	&\leq  32 d \frac{1}{n} + 104 \left(2^d-1\right) \frac{m}{n} +  2\left( 4d + 2r (2^d-1) + 36dr \right) \frac{1}{m} \nonumber \\
	&\quad + \norm[\infty]{\frac{1}{\abs{T_{m,n}}} \sum\limits_{t \in T_{m,n}} \frac{\evcf  (\L^r_m + t , \o) }{\abs{\L_m}} - \frac{\EE \admf (\L_m^r, \cdot)}{\abs{\L_m}} }.
\end{align}
for $n,m \in \NN$, $n>4m$ and $m>2r+1$.\\
The last term on the right side is precisely the point where we need to obtain a bound for the difference of a random function to a deterministic one. It is the norm of the difference of an averaged sample to an expected value, so at least for every fixed $x\in \RR$ the law of large numbers ensures
\begin{align*}
	\frac{1}{\abs{T_{m,n}}} \sum\limits_{t \in T_{m,n}} \frac{\evcf  (\L^r_m + t , \o) (x) }{\abs{\L_m}} \to  \frac{\EE \admf (\L_m^r, \cdot) (x)}{\abs{\L_m}} \quad \text{ for } n \to \infty
\end{align*}
for almost all $\o \in \O$. The challenge left is to sharpen the convergence to uniform bounds and quantify them.\\
To simplify notation we will use
\begin{align*}
	\sup_{x \in \RR}\abs{\frac{1}{\abs{T_{m,n}}} \sum\limits_{t \in T_{m,n}} \frac{\admf (\Lambda^r_m + t , \omega) (x) }{\abs{\Lambda_m}} - \EE \frac{\admf (\Lambda^r_m , \cdot) (x) }{\abs{\Lambda_m}}} = \sup\limits_{f \in \cF'}\abs{\frac{1}{\sumcount} \left(\sum\limits_{t=1}^\sumcount f(X_t)-\EE f(X_t)\right)}
\end{align*}
where $\cF'=\left\{\frac{\admf \left(\Lambda^r_m , \cdot\right)(x)}{\abs{\Lambda_m}} : x \in \RR\right\}$, $\sumcount=\abs{T_{m,n}}$ and $X_t = \left(\gamma_{t}\omega\right)_{\Lambda^r_m} $. Instead of a supremum over the real numbers we will instead treat the eigenvalue counting functions for each $x$ as a separate function from $\O$ to $\RR$ and then take the supremum over all of these functions.\\
The quantification we have in mind is of the form
\begin{align*}
	 \PP\left( \sup_{f\in \cF} \left| \frac{1}{\sumcount} \sum_{i=1}^\sumcount f(X_i) - \EE f(X_1) \right| \geq \kappa \right)
	\leq
	c_1 (\kappa , \sumcount) \cdot  \exp\left(c_2 (\kappa , \sumcount)\right)
\end{align*}
for all $\kappa >0$ with suitable $c_1 (\kappa , \sumcount)$ and $c_2 (\kappa , \sumcount)$.\\
Related to such concentration inequalities is the Orlicz norm.
\begin{ddef}[Orlicz norm] \label{def:orlicznorm}
	The \textbf{Orlicz norm} of a real random variable $X$ associated to a monotone increasing, non-constant convex function $\Phi\colon \RR \to \RR$ with $0 \leq \Phi (0)<1$ is
	\begin{equation*}
		\norm[\Phi]{X}:= \inf \left\{ C>0 : \EE \Phi \left( \frac{\abs{X}}{C}\right) \leq 1 \right\}
	\end{equation*}
	where $\inf (\emptyset):=\infty$.
\end{ddef}
The Orlicz norm for $\Phi(x)=x^p$ is the usual $L^p$-norm. If $\abs{X} \leq \abs{Y}$ almost surely, then $\norm[\Phi]{X} \leq \norm[\Phi]{Y}$. Of special interest are the Orlicz norms associated to the functions
\begin{align*}
		\psi_{p,M} (x)&:=\frac{1}{M} \mathrm{e}^{(x^p)}
\end{align*}
for $p\geq 1$ and especially $p=1$ and $p=2$. If $M=2$, the resulting Orlicz norm is identical to the one for
\begin{align*}
	\Psi_p (x)&:=\mathrm{e}^{(x^p)}-1\geq x^p.
\end{align*}		
In the rest of this proof we refer to these specific functions with $\psi_{p,M}$ and $\Psi_p$, while we use $\Phi$ for general monotone increasing, non-constant, convex functions with
$0 \leq \Phi (0)<1$.
\\
The connection to exponential concentration inequalities is made explicit by two bounds, 
which follow from Markov's inequality for convex functions and the layer cake formula, respectively. 
For a proof see e.g. \cite[Lemma 6.2]{Diss}.
\begin{llem}\label{lemma:orliczproperties} Let $\Phi$ be as in Definition \ref{def:orlicznorm}.
	\begin{enumerate}[label={(O\arabic*)}]
		\item If $\norm[\Phi]{X}\leq D < \infty$ then $\EE \Phi \left( \frac{\abs{X}}{\norm[\Phi]{X}} \right) \leq 1$ and $\PP (\abs{X}\geq y)\leq \frac{1}{\Phi \left(\frac{y}{D}\right)}$	\label{prop:orlicznormtoconcentration}
		\item If $X$ is a real random variable with $\PP(\abs{X}>x)\leq B\mathrm{e}^{-C x^p} \ \forall x>0$ with constants $B,C>0$ and $p\geq 1$ then
		\begin{align*}
			%	\norm[\Psi_p]{X}&\leq \left(\frac{1+B}{C}\right)^{1/p}\\
			\norm[\psi_{p,M}]{X}&\leq \left(\frac{B+M-1}{(M-1)C}\right)^{1/p}
		\end{align*}		\label{prop:orliczconcentrationtonorm}
%		\item \begin{align*}
%			\norm[L^p]{X} \leq \norm[\psi_{p,2}]{X},
%		\end{align*}
%		and thus
%		\begin{align*}
%			\EE \abs{X} = \norm[L^1]{X} \leq \norm[\psi_{1,2}]{X}.
%		\end{align*}\label{prop:orliczexpectationbound}
%		\item If $M\geq 2$ and $X=d$ almost surely, then
%		\begin{align*}
%			\norm[\psi_{p,M}]{X} &\leq \frac{\abs{c}}{(\log (M))^{1/p}}\leq \frac{\abs{c}}{(\log (2))^{1/p}}
%			%\norm[\Psi_p]{X} &\leq \frac{\abs{c}}{(\log (2))^{1/p}}
%		\end{align*}\label{prop:orliczconstantbound}
%		\item
%		%\begin{align*}
%		%	\norm[\Psi_1]{X} &\leq \frac{1}{\sqrt{\log (2)}} \norm[\Psi_2]{X}
%		%\end{align*}
%		Let $X$ and $Y$ be real random variables, then
%		\begin{align*}
%			\norm[\psi_{1,M}]{XY}\leq\norm[\psi_{2,M}]{X}\norm[\psi_{2,M}]{Y}.
%		\end{align*}
%		As a consequence we have
%		\begin{align*}
%			\norm[\psi_{1,M}]{X} &\leq \frac{1}{\sqrt{\log (M)}} \norm[\psi_{2,M}]{X}
%		\end{align*}\label{prop:orlicz12normbound}	
	\end{enumerate}
\end{llem}
Van der Vaart and Wellner (among others) established an estimate on the supremum of an (averaged) empirical process. However, they use a slightly different normalization, namely by $1/ \sqrt{\sumcount}$ instead of $1/\sumcount$. In the following we use their results with some arguments due to Pollard and adapt them to the present situation.
\begin{ddef}[{Empirical process (\cite[Chapter 2.1]{vanderWaartWellner1996})}]\label{def:empiricalprocess}
	Let $\sumcount \in \NN$ and let $X_1, X_2, ..., X_\sumcount$ be i.i.d.\ random variables taking values in a measurable space $Y$ with image measure $P$. Then the empirical process applied to a $P$-integrable function $f:Y \to \RR$  is
	\begin{equation}\label{eq:defempprocess}
		\GG_\sumcount (f):=\frac{1}{\sqrt{\sumcount}} \left( \sum_{i=1}^\sumcount \left(f(X_i) - \EE f(X_i)\right) \right).
	\end{equation}
	Let $\cF$ be a set of $P$-integrable functions $f:Y \to \RR$ then the map $\cF \ni f \to \GG_\sumcount (f)$ is called the $\cF$-indexed empirical process. 
\end{ddef}
The map  $f\mapsto \GG_\sumcount (f)$ is linear. If $\cF$ is finite, the following bound holds.
\begin{llem}\label{lemma:expfinitesup}
	Let $\GG_\sumcount$ be as in Definition \ref{def:empiricalprocess}. Let $\cF$ be a finite set of measurable bounded functions. Then for $M\geq 2$
	\begin{equation} \label{eq:expfinitesup}
		\norm[\psi_{1,M}]{\max\limits_{f \in \cF} \abs{\GG_\sumcount(f)}} \leq K_{\psi, M}' \left(\max\limits_{f \in \cF} \frac{\norm[\infty]{f}}{\sqrt{\sumcount}}(2+\log (\abs{\cF}))+\max\limits_{f \in \cF} \norm[L^2(P)]{f} \sqrt{2+\log (\abs{\cF})}  \right)
	\end{equation}
	where $K_{\psi, M}'=\frac{4(M+1)}{\log(3/2)(M-1)}$.
\end{llem}
The qualitative idea for the proof is due to Lemma 2.2.10 in \cite{vanderWaartWellner1996} by van der Vaart and Wellner. The proof of the quantitative version can be found in 
\cite[Lemma 6.4]{Diss}.
We give a short sketch here.
\begin{proof}[Sketch of the proof]
The proof mainly rests on two lemmas.
The first is an adaption of Lemma 3.2 from \cite{Pollard90} that gives a bound for the Orlicz norm of a supremum over an empirical process.
For $p=1,2$ and real random variables $X_1,...,X_m$, the inequality
	\begin{align}\label{eq:finitemaxauxpsi}
		\norm[\psi_{p,M}]{\max\limits_{1\leq i \leq m} X_i} \leq \frac{(2+\log (m))^{1/p}}{\log (3/2)} \max\limits_{1\leq i \leq m} \norm[\psi_{p,M}]{X_i}.
	\end{align}
holds.
The second is the Bernstein inequality, which states that for i.i.d. random variables $X_1, X_2,...,X_\sumcount$ with $\abs{X_i}\leq C$ for all $1\leq i \leq \sumcount$, the inequality
	\begin{equation*}
		\PP \left(\sum_{i =1}^\sumcount \left( X_i - \EE X_i \right) \geq x\right) \leq \mathrm{e}^{-\frac{1}{2}\frac{x^2}{\sigma^2 + c x}}
	\end{equation*}
	with $\sigma^2:=\sum_{i =1}^\sumcount \EE \left(X_i^2\right)$ and $c=\frac{C}{3}$ holds for all $x>0$. Applied to empirical processes the Bernstein inequality leads to
	\begin{align*}
		\PP\left(\abs{\GG_\sumcount (f)}\indic \left\{\abs{\GG_\sumcount (f)} \leq \frac{b}{a}\right\}>x\right) &\leq 2 \mathrm{e}^{-\frac{x^2}{4b}} \\
		\PP\left(\abs{\GG_\sumcount (f)}\indic \left\{\abs{\GG_\sumcount (f)} > \frac{b}{a}\right\}>x\right) &\leq 2 \mathrm{e}^{-\frac{x}{4a}}.
	\end{align*}
	where
	\begin{align*}
		a:=\max\limits_{f \in \cF} \frac{2}{3} \frac{ \norm[\infty]{f}}{\sqrt{\sumcount}}\leq \max\limits_{f \in \cF} \frac{ \norm[\infty]{f}}{\sqrt{\sumcount}} , \ b:= \max\limits_{f \in \cF} \EE (f(X_1))^2 =\max\limits_{f \in \cF} \norm[L^2(P)]{f}^2.
	\end{align*}
	\ref{prop:orliczconcentrationtonorm} implies now
	\begin{align}\label{eq:supremumoftruncatedempiricalprocess}
		\norm[\psi_{2,M}]{\abs{\GG_\sumcount (f)}\indic \left\{\abs{\GG_\sumcount (f)} \leq \frac{b}{a}\right\}} &\leq \sqrt{\frac{4(M+1)}{M-1} b} \\
		\norm[\psi_{1,M}]{\abs{\GG_\sumcount (f)}\indic \left\{\abs{\GG_\sumcount (f)} > \frac{b}{a}\right\}} &\leq \frac{4(M+1)}{M-1}  a . \nonumber
	\end{align}
	Since the right side of both inequalities only depends on $a$ and $b$, not $f$, it is now possible to bound suprema over Orlicz norms of this type. In combination with \eqref{eq:finitemaxauxpsi} this allows bounds on the Orlicz norm of the maximum of truncated empirical processes like those on the left sides of \eqref{eq:supremumoftruncatedempiricalprocess} indexed by a finite set. By using some other properties of Orlicz norms it is possible to split maxima of empirical processes into two truncated parts, leading to the bound stated in the lemma.
\end{proof}
Next we extend the bound to countable sets of functions. We do this by approximating the set with ever larger finite sets, but for this strategy to succeed we need some regularity of the considered functions. The idea here is to pick for any $q \in \NN$ a collection of functions $\left(g^q_i\right)_{1\leq i \leq N(q)}$ that is getting progressively denser and ensures that for every $f \in \cF$ there are indices $1\leq i(f,q) \leq N(q)$ such that 
\[
f=\sum\limits_{q=2}^\infty \left(g^q_{i(f,q)}-g^{q-1}_{i(f,q-1)}\right) 
\quad \text{ and } \quad 
\norm[L^2(P)]{g^q_{i(f,q)}-g^{q-1}_{i(f,q-1)}}<2^{-q}
\]
is always true, 
meaning that the functions $\left(g^q_i\right)_{1\leq i \leq N(q)}$ at \textquote{level} $q$ can be used to approximate any $f \in \cF$ up to an $L^2 (P)$-error of $2^{-q+1}$.
\\

As a consequence of Lemma \ref{lemma:expfinitesup} we need to control both the error in the $L^2 (P)$- and the sup-norm, which complicates the arguments, in addition to the challenge of finding the $g^q_i$. For the sake of illustration of the procedure let us assume for a moment that we have for finite $\cF$ a bound like
	\begin{align*}
	\norm[\psi_{1,M}]{\max\limits_{f \in \cF} \abs{\GG_\sumcount (f)}} \leq   \left(\max\limits_{f \in \cF} \norm[L^2]{f} \right) \sqrt{2+\log ( \abs{\cF})}
\end{align*}
instead the actual bound \eqref{eq:expfinitesup} of Lemma \ref{lemma:expfinitesup} and were able to select $g^q_i$ as above. Then we can try to treat countable sets $\cF$ in this hypothetical simplified situation by
\begin{align*}
	\sup_{f\in \cF} \abs{\GG_\sumcount (f)}=\sup_{f\in \cF} \abs{\GG_\sumcount \left(\sum\limits_{q=2}^\infty g^q_{i(f,q)}-g^{q-1}_{i(f,q-1)}\right)} 
	\leq \sum\limits_{q=2}^\infty \sup_{f\in \cF} \abs{\GG_\sumcount \left( g^q_{i(f,q)}-g^{q-1}_{i(f,q-1)}\right)}.	
\end{align*}
Since we only have finitely many choices for each $q \in \NN$, the supremum over $f\in \cF$ for each term of the sum is bounded by supremum over the functions $\left(g^q_i\right)_{1\leq i \leq N(q)}$. This is a supremum over a finite set, and since we assumed that $\norm[L^2(P)]{g^q_{i(f,q)}-g^{q-1}_{i(f,q-1)}}<2^{-q}$ is always true, our assumed toy model above would lead to
\begin{align*}
	\norm[\psi_{1,M}]{\max\limits_{f \in \cF} \abs{\GG_\sumcount (f)}} \leq \sum\limits_{q=2}^\infty 2^{-q} \sqrt{2+\log ( \abs{\cF_q})}
\end{align*}
where $|\cF_q|$ is the number of possible combinations $\left( g^q_{i(f,q)}-g^{q-1}_{i(f,q-1)}\right)$. If this sum converges we would get the bound that we want.\\
Unfortunately, Lemma \ref{lemma:expfinitesup} is more complicated than this hypothetical simplified model, but the general idea of replacing every function by a sum of small \textquote{links} with a choice from a finite set remains valid, albeit it is necessary to do some extra steps. The way we identify the functions $g^q_i$ described above and measure how many \textquote{links} we need for a given set of functions $\cF$ will be its bracketing number.
\begin{ddef}[Bracketing Cover]\label{definition:bracketingnumber}
	Let $Y$ be a measurable space and let $P$ be a measure on $Y$. For two functions $l, u \colon Y \to \RR$ with $l(y)\leq u (y)$ for all $y \in Y$ the \textbf{bracket} $[l,u]$ is the set
	\begin{equation*}
		[l,u] =  \{f \colon Y \to \RR \mid l(y) \leq f(y)\leq u(y) \text{ for all }y \in Y \}.
	\end{equation*}
	Here, $l$ is called the \textbf{lower boundary function} and $u$ the \textbf{upper boundary function} of the bracket.\\
	Let $\cF$ be a set of functions $f\colon Y \to \RR$. Then $\cI=\left\{ [l_i, u_i], 1\leq i \leq N\right\}$ is a \textbf{monotone bracketing cover} of $\cF$ if
	\begin{enumerate}[label={(BC\arabic*)}]
		\item \label{prop:bracketingmonotonicity} For all $y \in Y$ and $1\leq i \leq N$ we have
		\begin{align*}
			u_{i}(y) \geq l_{i} (y)
		\end{align*}
		and for all $y \in Y$ and $1\leq i \leq N-1$ we have
		\begin{align*}
			l_{i+1} (y) \geq u_{i}(y).
		\end{align*}		
		\item \label{prop:bracketingcover}$\cF \subset \cup_{i=1}^{N} [l_{i},u_{i}]$.
	\end{enumerate}
	A sequence $\left(\cI_q\right)_{q \in \NN}$ with $\cI_q=\left\{ [l_{q,i}, u_{q,i}], 1\leq i \leq N_{[]}(q, \cF, P)\right\}$ of monotone bracketing covers with sequences of boundary functions $\left(l_{q,i}\right)_{1\leq i \leq N_{[]}(q, \cF, P)}$ and $\left(u_{q,i}\right)_{1\leq i \leq N_{[]}(q, \cF, P)}$ in $L^2 (P)$ is a \textbf{nested monotone bracketing cover} (with regard to $P$) if
	\begin{enumerate}[label={(NBC\arabic*)}]
		\item \label{prop:bracketingcoversize}
		\begin{align*}
			\norm[L^2(P)]{u_{q,i}-l_{q,i}}<2^{-q}
		\end{align*}
		for all $q \in \NN$, $1\leq i \leq N_{[]}(q, \cF, P)$,
		\item \label{prop:bracketingnested}For all $q,t \in \NN$ and $1\leq i \leq N_{[]}(q, \cF, P)$ we have
		\begin{align*}
			l_{q,i} &\in \left(l_{q+t,j}\right)_{1\leq j \leq N_{[]}(q+t, \cF, P)}\\
			u_{q,i} &\in \left(u_{q+t,j}\right)_{1\leq j \leq N_{[]}(q+t, \cF, P)}.
		\end{align*}
	\end{enumerate}
	We call the function $q \to N_{[]}(q, \cF, P)$ the \textbf{monotone bracketing function} of the nested monotone bracketing cover.
\end{ddef}
Note that this definition is more restricted than bracketings used in many other cases, which usually only require \ref{prop:bracketingcover} and \ref{prop:bracketingcoversize}. These are the crucial properties for the next theorem as well, but the other properties naturally hold for the eigenvalue counting functions we are interested in and lead to better constants here.

A rigorous analogon of the discussion above of the hypothetical simplified situation is the following:

\begin{tthm}[Orlicz norm of countable set supremum of the empirical process]\label{theorem:expofcountablesup}
	Let $X_1, X_2, ...$ be i.i.d. random variables taking values in a measurable space $Y$, with image measure $P$, 
	$\cF$ a countable set of measurable functions $f:Y \to [0,1]$ with a nested monotone bracketing cover, and let $N_{[]}(q, \cF, P)$ be the associated monotone bracketing function. 
	%of the nested monotone bracketing cover. 
	\\
	Then for all $M\geq 2$
	\begin{align}\label{eq:orlicznormcountablesup}
		\norm[\psi_{1,M}]{ \sup_{f\in \cF} \left| \frac{1}{\sqrt{\sumcount}} \left( \sum_{i=1}^\sumcount \left(f(X_i) - \EE f(X_i)\right) \right) \right|} \leq K_{\psi, M}'' \sum\limits_{q=0}^\infty 2^{-q} \sqrt{2+ \log \left(N_{[]}(q, \cF, P)\right)}
	\end{align}
	holds, where $\norm[\psi_{1,M}]{\cdot}$ is the Orlicz norm associated to $\psi_{1,M} := \frac{1}{M}\mathrm{e}^x$, $K_{\psi, M}'=\frac{4(M+1)}{\log (3/2) (M-1)}$ and $K_{\psi, M}''=(10 \ K_{\psi, M}'+\frac{4}{\log (M)})$. 
\end{tthm}
The qualitative version of this result was established in Theorems 2.5.6 and 2.14.2 of \cite{vanderWaartWellner1996}, the quantitative version and its proof can be found in \cite[Theorem 6.9]{Diss}.
\subsection{Bracket coverings for eigenvalue counting functions}\label{subsec:bracketcoverings}
The next step now is to find a nested monotone bracketing cover for (a countable subset of) the set of functions
\begin{align*}
	\cF'=\left\{\frac{\admf (\Lambda^r_m,\cdot)(x)}{\abs{\Lambda_m}} : x \in \RR\right\}.
\end{align*}
To do this, we will use some properties of the normalized eigenvalue counting functions, namely the boundedness \ref{prop:admissibleboundedness}, the asymptotic behaviour
\begin{align*}
	\lim\limits_{x \to - \infty} \frac{\admf (\Lambda^r_m, \o)(x)}{\abs{\Lambda_m}}=0 
	\quad \text{ and }
	\lim\limits_{x \to \infty} \frac{\admf (\Lambda^r_m, \o) (x)}{\abs{\Lambda_m}}=\frac{\abs{\Lambda^r_m}}{\abs{\Lambda_m}} \quad \text{ for all } \o \in \O,
\end{align*}
the monotonicity \ref{prop:admissiblemonotone}, and the point-wise measurability \ref{prop:measurable}.\\% and finally the fact that eigenvalue counting functions of Anderson Operators on $\ZZ^d$ are continuous \cite{DelyonSouillard84}.\\
We further define
\begin{align*}
	\mrcF \colon \RR \to \RR, \qquad \mrcF(x) := \EE\left(\frac{\admf (\Lambda^r_m, \cdot) (x)}{\abs{\Lambda_m}}\right)
\end{align*}
which inherits the boundedness, limit behaviour and monotonicity by dominated convergence. Because of the limit behaviour it is convenient to define
\begin{align*}
	\mrcF (-\infty)=0, \quad 
	\mrcF^- (\infty)=\frac{\abs{\Lambda^r_m}}{\abs{\Lambda_m}}.
\end{align*}
We further define
\begin{align*}
	\mrcF^{-1}(\alpha) := \inf \{ \lambda\in\RR \mid \mrcF (\lambda) \geq  \alpha \},
\end{align*}
the \textbf{quantile function} or  \textbf{generalized inverse} of $\mrcF$, and based on this
\begin{align*}
	x_j (\eps) &:= \begin{cases}
		-\infty & \text{ if } j=0 \\
		\mrcF^{-1} ( j \cdot \eps \cdot \frac{\abs{\Lambda^r_m}}{\abs{\Lambda_m}})   &\text{ if } 0<j<k\\
		\infty & \text{ if } j=k
	\end{cases}
\end{align*}
for $\eps \in \QQ \cap (0,1)$ and $j=0,1,\dots ,k:=\lceil \frac1{\eps} \rceil$.\\
With these definitions we find the following nested monotone bracketing cover.
\begin{llem}\label{lem:ourFnmbc}
	The brackets
	\begin{align*}
		\cF_{2^{-q},j}&:=\left[\frac{\admf (\Lambda^r_m, \cdot)}{\abs{\Lambda_m}}\left(x_{j-1} \left(2^{-2q}\right)\right),\frac{\admf (\Lambda^r_m, \cdot)}{\abs{\Lambda_m}} \left(x_j \left(2^{-2q}\right)\right)\right],
\quad j \in \{1,\dots,2^{2q}\}	\end{align*}
	 form a nested monotone bracketing cover of
	\begin{align}\label{eq:countablesetformonrcontboundfct}
		\cF :=
		\left\{\frac{\admf (\Lambda^r_m, \cdot)}{\abs{\Lambda_m}}(x):\O_\L \to\RR \mid \exists Q \in \mathbb{Q}: x=\mrcF^{-1} (Q) \right\}
	\end{align}
	with regard to the measure $\PP$ and monotone bracketing function
	\begin{align*}
		N_{[]}(q, \cF, \PP)\leq 2^{2q}.
	\end{align*}
\end{llem}
For an illustration, see Figure \ref{fig:bracketingcover}. The left diagram shows how $x_j (\eps)$ is chosen based on $\mrcF$, while the right diagram is an illustration of the bracketing that results. 
On the right the horizontal axis corresponds to an abstracted space of possible $\o$ and the functions depicted are some `almost normalized´ eigenvalue counting functions $\o \mapsto\frac{N(\Lambda^r_m , \omega) (x)}{\abs{\Lambda_m}}$ for different $x$. Each colored dotted area denotes one bracket (with the color of its upper boundary function) and covers all possible function values for functions in this bracket.
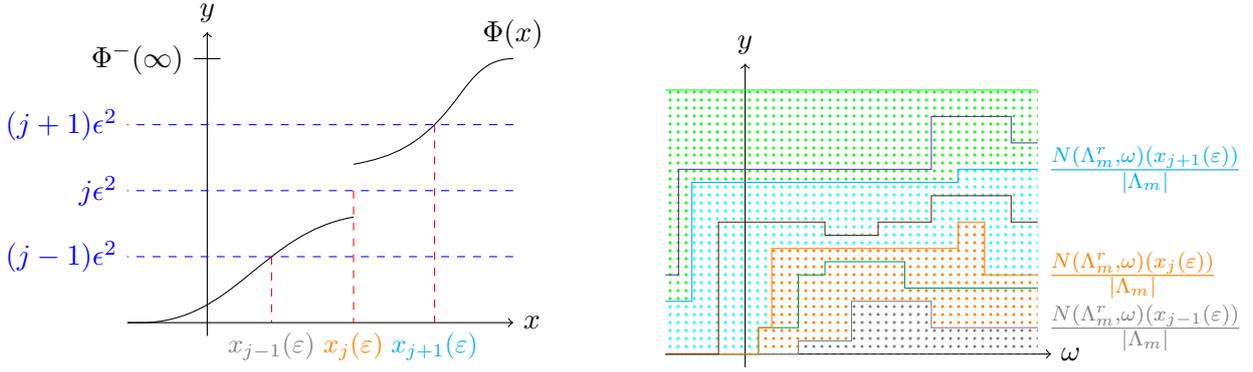
\begin{figure}
\begin{subfigure}{0.45\textwidth}
	\centering
		\begin{tikzpicture}
			\begin{scope}[scale=0.35]
				\draw[->, name path=xachse] (-3,0) -- (11.5,0) node[right] {$x$};
				\draw[->] (0,-0.5) -- (0,11) node[above] {$y$};
				\draw[blue,dashed,name path=linie1] (11.5,2.5)--(-3,2.5) node[left] {$(j-1) \epsilon^2$};
				\draw[blue,dashed,name path=linie2] (11.5,5)--(-3,5) node[left] {$j \epsilon^2$};
				\draw[blue,dashed,name path=linie3] (11.5,7.5)--(-3,7.5) node[left] {$(j+1) \epsilon^2$};
				\draw (0.5,10)--(-0.5,10) node[left]{$\mrcF^- (\infty)$};
				%\draw[rounded corners] (-3,0) .. controls(-1,0).. (2,3) .. controls(9,10) .. (10,10);
				\draw[name path=kurve1] (-3,0) to (-2.5,0) to [out=0,in=190] (5.5,4);
				\draw[name path=kurve2] (5.5,6) to [out=10,in=230] (9,8) to [out=50, in=180] (11.5,10) node[above] {$\mrcF (x)$};
				\draw [red, dashed,name intersections={of=linie1 and kurve1, by={schnittpunkt1}}] (schnittpunkt1)--($(-3,0)!(schnittpunkt1)!(11,0)$) node[below] {\color{gray}\small{$x_{j-1}(\eps)$}};
				\draw [red, dashed] (5.5,5)--($(-3,0)!(5.5,5)!(11,0)$) node[below] {\color{orange}\small{$x_{j}(\eps)$}};
				\draw [red, dashed,name intersections={of=linie3 and kurve2, by={schnittpunkt3}}] (schnittpunkt3)--($(-3,0)!(schnittpunkt3)!(11,0)$) node[below] {\color{cyan}\small{$x_{j+1}(\eps)$}};
				%\draw[name path=kurve] (-3,0) to (-2,0) to [out=0,in=220] (2,3) to [out=40,in=230] (9,8) to [out=60, in=180] (11.5,10) node[above] {$\EE \frac{N(\Lambda^r_m , \omega) (x)}{\abs{\Lambda_m}}$};
				%\draw [red, dashed,name intersections={of=linie1 and kurve, by={schnittpunkt1}}] (schnittpunkt1)--($(-3,0)!(schnittpunkt1)!(11,0)$) node[below] {$x_{j-1}(\eps)$};
				%\draw [red, dashed,name intersections={of=linie2 and kurve, by={schnittpunkt2}}] (schnittpunkt2)--($(-3,0)!(schnittpunkt2)!(11,0)$) node[below] {$x_{j}(\eps)$};
				%\draw [red, dashed,name intersections={of=linie3 and kurve, by={schnittpunkt3}}] (schnittpunkt3)--($(-3,0)!(schnittpunkt3)!(11,0)$) node[below] {$x_{j+1}(\eps)$};
				%\foreach \y in {1,2,3} \draw [red, dashed,name intersections={of=linie\y and kurve, by={schnittpunkt\y}}] (schnittpunkt\y)--($(-3,0)!(schnittpunkt\y)!(11,0)$);
			\end{scope}
		\end{tikzpicture}
\end{subfigure}
\hfill
\begin{subfigure}{0.45\textwidth}
	\centering
		\begin{tikzpicture}
			\begin{scope}[scale=0.35]
				\draw[->, name path=xachse] (-3,0) -- (11.5,0) node[right] {$\omega$};
				\draw[->] (0,-0.5) -- (0,11) node[above] {$y$};
				\fill[pattern=dots, pattern color=green] (-3,10)--(11,10)--(11,7)--(8,7)--(8,6.5)--(-2,6.5)--(-2,2)--(-3,2)--cycle;
				\fill[pattern=dots, pattern color=cyan] (-3,2) -- (-2,2) -- (-2,6.5)--(8,6.5)--(8,7)--(11,7)--(11,3)--(9,3)--(9,5)--(8,5)--(8,4)--(1,4)--(1,1)--(0.5,1)--(0.5,0)--(0,0)--(-3,0)--cycle;
				\fill[pattern=dots, pattern color=orange] (-3,0) -- (0,0) -- (0.5,0) -- (0.5,1) -- (1,1) -- (1,4) -- (8,4) -- (8,5) -- (9,5) -- (9,3) -- (11,3) -- (11,1) -- (7,1) --(7,2)--(4,2)--(4,0.5)--(2,0.5)--(2,0)--(0,0)--(-3,0)--cycle;
				\fill[pattern=dots, pattern color=gray] (-3,0) -- (0,0) -- (2,0)--(2,0.5)--(4,0.5)--(4,2)--(7,2)--(7,1)--(11,1)--(11,0)--cycle;
				\draw[gray] (-3,0) -- (0,0) -- (2,0)--(2,0.5)--(4,0.5)--(4,2)--(7,2)--(7,1)--(11,1)node[right] {\color{gray}$\frac{N(\Lambda^r_m , \omega) (x_{j-1}(\eps))}{\abs{\Lambda_m}}$};
				\draw[ForestGreen] (-3,0) -- (0,0) -- (0.5,0) -- (0.5,1) -- (2,1)--(2,3)--(3,3)--(3,3.5)--(6,3.5)--(6,2.5)--(11,2.5);
				\draw[orange] (-3,0) -- (0,0) -- (0.5,0) -- (0.5,1) -- (1,1)--(1,4)--(8,4)--(8,5)--(9,5)--(9,3)--(11,3)node[right] {\color{orange}$ \frac{N(\Lambda^r_m , \omega) (x_{j}(\eps))}{\abs{\Lambda_m}}$};
				\draw[Brown] (-3,0) -- (-1,0) -- (-1,5)--(3,5)--(3,4.5)--(5,4.5)--(5,5) -- (7,5)--(7,6)--(10,6)--(10,5)--(11,5);
				\draw[cyan] (-3,2) -- (-2,2) -- (-2,6.5)--(8,6.5)--(8,7)--(11,7)node[right] {\color{cyan}$\frac{N(\Lambda^r_m , \omega) (x_{j+1}(\eps))}{\abs{\Lambda_m}}$};
				\draw[Violet] (-3,3) -- (-2.5,3) -- (-2.5,7)--(7,7)--(7,9)--(10,9)--(10,8)--(11,8);
				\draw[green] (-3,10)--(11,10);
			\end{scope}
		\end{tikzpicture}
\end{subfigure}
\caption{Illustration of the choice of $x_j (\eps)$ (left) and a nested monotone bracketing cover consisting of evcf (reight).\label{fig:bracketingcover}}
\end{figure}
\begin{proof}
	We first prove \ref{prop:bracketingcover} and \ref{prop:bracketingcoversize} of Definition \ref{definition:bracketingnumber}, but to avoid technical complications we will only consider the case where $\mrcF$ is continuous. This already covers a large number of cases due to so called Wegner estimates, see for example \cite[Chapter 5.5]{Kirsch}. 
	The general proof without a continuity assumption is technically more involved but follows the same ideas. It can be found in \cite[Lemma 6.16]{Diss}.\\
	All functions $\frac{\admf (\Lambda^r_m, \cdot)}{\abs{\Lambda_m}} \left(x_j \left(2^{-2q}\right)\right)$ are contained in at least on bracket. For every $x\in \mrcF^{-1} (\mathbb{Q})\setminus\{x_0  \left(2^{-2q}\right),\dots,x_k  \left(2^{-2q}\right)\}$ there is a unique $j \in \{1,\dots,2^{2q}\}$ such that $x_{j-1} \left(2^{-2q}\right)< x <x_j \left(2^{-2q}\right)$. Thus
	\begin{align*}
		\frac{\admf (\Lambda^r_m, \o)}{\abs{\Lambda_m}} \left(x_{j-1} \left(2^{-2q}\right)\right) \leq \frac{\admf (\Lambda^r_m, \o)}{\abs{\Lambda_m}} (x) \leq \frac{\admf (\Lambda^r_m, \cdot)}{\abs{\Lambda_m}} \left(x_j \left(2^{-2q}\right)\right)\ \forall \o \in \O_\L
	\end{align*}
	because of the monotonicity of $\admf (\Lambda^r_m, \o) (x)$ in the variable $x$, leading to
	\begin{align*}
		 \frac{\admf (\Lambda^r_m, \cdot)}{\abs{\Lambda_m}} (x) \in  \cF_{\eps,j}
	\end{align*}
	hence
	\begin{align*}
		\cF \subset \bigcup_{j=1}^{2^{2q}} \cF_{2^{-2q},j}.
	\end{align*}
	Next observe
	\begin{align}\label{eq:F-inverse relations}
		\mrcF(\mrcF^{-1}(\alpha)) = \alpha
	\end{align}
	by the assumed continuity of $\mrcF$. %\alert{Nein! Evtl für Produktmaße mit stetiger Verteilung}
	The normalized eigenvalue counting functions only take values in $[0,1]$ and thus
	\begin{align*}
		&\EE\left(\left(\frac{\admf (\Lambda^r_m, \cdot)}{\abs{\Lambda_m}} \left(x_j \left(2^{-2q}\right)\right) - \frac{\admf (\Lambda^r_m, \cdot)}{\abs{\Lambda_m}} \left(x_{j-1} \left(2^{-2q}\right)\right))\right)^2\right)\\
		&\quad\quad\leq \EE\left(\frac{\admf (\Lambda^r_m, \cdot)}{\abs{\Lambda_m}} \left(x_j \left(2^{-2q}\right)\right) - \frac{\admf (\Lambda^r_m, \cdot)}{\abs{\Lambda_m}} \left(x_{j-1} \left(2^{-2q}\right)\right)\right) \\
		&\quad\quad= \mrcF \left(x_j \left(2^{-2q}\right)\right) - \mrcF \left(x_{j-1} \left(2^{-2q}\right)\right) \\
		&\quad\quad= j2^{-2q} \frac{\abs{\Lambda^r_m}}{\abs{\Lambda_m}} - (j-1)2^{-2q}\frac{\abs{\Lambda^r_m}}{\abs{\Lambda_m}} \leq 2^{-2q} %\qedhere
	\end{align*}
	holds, with the consequence that the $L^2 (\PP)$-size of every bracket is less or equal to $2^{-q}$. Condition \ref{prop:bracketingmonotonicity} is automatically fulfilled by the monotonicity of $\admf (\Lambda^r_m, \o)$ and $x_{j-1}\left(2^{-2q}\right) \leq x_j \left(2^{-2q}\right)$. By the definition of $x_j \left(2^{-2q}\right)$ we have $x_j ((2^{-q})^2)=x_{4j}((2^{-(q+1)})^2)$ for all $q \in \NN$ and $0\leq j \leq 2^{2q}$, which ensures condition \ref{prop:bracketingnested}. Every bracket for $\eps=2^{-q}$ contains exactly four brackets for $\eps=2^{-(q+1)}$.\\
	We need at most $k = \lceil \frac{1}{2^{-2q}} \rceil$ sets to cover $\cF$ in the described way, so we have
	\begin{align*}
		N_{[]}(q, \cF, \PP)\leq \lceil \frac{1}{2^{-2q}} \rceil=2^{2q}.
	\end{align*}
\end{proof}
\subsection{Concentration inequalities}
With the bound on the bracketing function established in Lemma \ref{lem:ourFnmbc}, Theorem \ref{theorem:expofcountablesup} applied to evcfs gives a concentration inequality.
\begin{cor}[Sub-root-exponential concentration inequality]\label{cor:firstconcineq}
	Let $\evcf$ be the eigenvalue counting function as defined in \eqref{eq:defevcf}, and let $\L^r_{m,1},...,\L^r_{m,\sumcount}$ be translates of $\Lambda^r_m =\big([0,m)^d\cap \ZZ^d\big)^r\neq \emptyset$ such that $\min \{ \d_{\text{set}}(\L^r_{m,i},\L^r_{m,j}) \mid i \neq j \}>r$.
	%In the setting of Theorem \ref{theorem:expofcountablesup} with the additional restriction $N_{[]}(q, \cF, P) \leq V 2^{W q}$ for some $V,W\geq 1$ we have
	Then for $M\geq 2, \ \kappa >0$
	\begin{align*}
		\PP\left( \sup_{x\in \mrcF^{-1} (\QQ)} \left| \frac{1}{\sumcount} \sum_{i=1}^\sumcount \frac{\admf (\Lambda^r_{m,i}, \o) (x)}{\abs{\Lambda_m}} - \EE \frac{\admf (\Lambda^r_m, \o)(x)}{\abs{\Lambda_m}} \right| \geq \kappa \right)
		\leq M \exp \left(- \frac{\sqrt{\sumcount} \kappa}{K_M}\right)
	\end{align*}	
	holds, where
	\begin{align*}
		K_M=\left(10\frac{4(M+1)}{\log (3/2) (M-1)}+\frac{4}{\log (M)}\right)\sum\limits_{q=0}^\infty 2^{-q} \sqrt{2+ + 2 q \log (2)}<\infty.
	\end{align*}
\end{cor}
\begin{proof}
	We first define new functions $\admf_{\L^r_m} \colon \O_{\L^r_m} \to \BB$ via
	\begin{align*}
		\admf_{\L^r_m}(\nu)(x):=\admf (\Lambda^r_m, i_{\L^r_m} (\nu)) (x)
	\end{align*}
	for $\nu \in \O_{\L^r_m}$ and $x \in \RR$, with the restricted probability space $\O_{\L^r_m}$ as defined in \ref{prop:M3_independence}. Here $\BB$ is the space of bounded right-continuous functions.\\
	For fixed $\nu$ we still have $\admf_{\L^r_m}(\nu)(x) \leq \admf_{\L^r_m}(\nu)(y)$ if $x \leq y$. Furthermore by \ref{prop:admissiblelocality} we have
	\begin{align}\label{eq:relrestrevcftoevcf}
		\admf_{\L^r_m}(\Pi_{\L^r_m} (\o))(x)=\admf (\Lambda^r_m, i_{\L^r_m} (\Pi_{\L^r_m} (\o))) (x)=\admf (\Lambda^r_m, \o) (x)
	\end{align}		
	 for all $\o \in \O$ and thus $\EE_{\PP_{\L^r_m}}\admf_{\L^r_m}(\cdot)(x)=\EE \admf (\Lambda^r_m, \cdot) (x)$. As a consequence the bracketing established in Lemma \ref{lem:ourFnmbc} carries over to
	\begin{align*}
		\tilde{\cF}:=
		\left\{\frac{\admf_{\L^r_m}(\cdot)}{\abs{\Lambda_m}}(x):\O_\L \to\RR \mid \exists Q \in \mathbb{Q}: x=\mrcF^{-1} (Q) \right\}
	\end{align*}
	with
	\begin{align*}
		N_{[]}(q, \tilde{\cF}, \PP_{\L^r_m})\leq \lceil \frac{1}{2^{-2q}} \rceil=2^{2q}.
	\end{align*}
	Now let $z_i \in \ZZ^d$ be such that $\L^r_{m,i}=\L^r_m+z_i$. Then by \ref{prop:M1_translationinvariance} we have
	\begin{align}\label{eq:relshiftedevcf}
		\admf (\Lambda^r_{m,i}, \o) (x)=\admf (\L^r_m,\gamma_{z_i} \o)(x)=\admf_{\L^r_m}(\Pi_{\L^r_m} (\gamma_{z_i} \o))(x).
	\end{align}
	Note that since
	\begin{align*}
		\Pi_{\L^r_m+z}(\o)=(\o_y)_{y \in \L^r_m+z}
	\end{align*}
	and
	\begin{align*}
		\Pi_{\L^r_m} (\gamma_z \o)=(\o_{y+z})_{y \in \L^r_m}
	\end{align*}
	both $\Pi_{\L^r_m+z}(\o)$ and $\Pi_{\L^r_m} (\gamma_z \o)$ contain the same values, just differently indexed. Therefore, $\left(\Pi_{\L^r_m}\circ \gamma_{z_i}\right)_{1\leq i\leq \sumcount}$ are independent and identically distributed by \ref{prop:M1_translationinvariance} and \ref{prop:M3_independence}.\\
	We now apply Theorem \ref{theorem:expofcountablesup} for the set $\tilde{\cF}$ with the i.i.d. random variables $X_i:=\Pi_{\L^r_m} \circ \gamma_{z_i}$, leading to
	\begin{multline*}
		\norm[\psi_{1,M}]{ \sup_{x \in \mrcF^{-1} (\QQ)} \left| \frac{1}{\sqrt{\sumcount}} \left( \sum_{i=1}^\sumcount \left(\frac{\admf_{\L^r_m}(\Pi_{\L^r_m} \circ \gamma_{z_i})(x)}{\abs{\Lambda_m}} - \frac{\EE_{\PP_{\L^r_m}}\admf_{\L^r_m}(\cdot)(x)}{\abs{\Lambda_m}}\right) \right) \right|}\\
		\quad \quad \leq K_{\psi, M}'' \sum\limits_{q=0}^\infty 2^{-q} \sqrt{2+ 2q \log (2)}
	\end{multline*}
	for all $M\geq 2$ and $K_{\psi, M}''$ as in Theorem \ref{theorem:expofcountablesup}. Substituting back via \eqref{eq:relrestrevcftoevcf} and \eqref{eq:relshiftedevcf} leads to
	\begin{multline*}
		\norm[\psi_{1,M}]{ \sup_{x \in \mrcF^{-1} (\QQ)} \left| \frac{1}{\sqrt{\sumcount}} \left( \sum_{i=1}^\sumcount \left(\frac{\admf (\Lambda^r_{m,i}, \o) (x)}{\abs{\Lambda_m}} - \EE \frac{\admf (\Lambda^r_m, \o)(x)}{\abs{\Lambda_m}}\right) \right) \right|} \\
		\quad \quad \leq K_{\psi, M}'' \sum\limits_{q=0}^\infty 2^{-q} \sqrt{2+ 2q \log (2)}=:K_M.
	\end{multline*}
	At last we use \ref{prop:orlicznormtoconcentration} to establish for all $\eta>0$
	\begin{align*}
		\PP \left(\sup_{x \in \mrcF^{-1} (\QQ)} \left| \frac{1}{\sqrt{\sumcount}} \left( \sum_{i=1}^\sumcount \left(\frac{\admf (\Lambda^r_{m,i}, \o) (x)}{\abs{\Lambda_m}} - \EE \frac{\admf (\Lambda^r_m, \o)(x)}{\abs{\Lambda_m}}\right) \right) \right| \geq \eta \right) \leq \frac{1}{\psi_{1,M} \left(\frac{\eta}{K_M}\right)}=M \mathrm{e}^{-\frac{\eta}{K_M}}.
	\end{align*}
	Using $\eta=\kappa \sqrt{\sumcount}$ then leads to the stated inequality.
\end{proof}
We can further improve the concentration inequality from an exponential of the square root of $\sumcount$ to a sub-exponential of $\sumcount$ by using the following concentration inequality due in this form to Massart.
\begin{tthm}[{\cite[Equation (5.45)]{Massart2007}}]\label{theorem:Massart}
	Let $T$ be a countable set and let $Z_1, Z_2,...,Z_\sumcount$ be independent random vectors taking values in $\RR^T$ with $\EE (Z_{i,t})=0$ and $\abs{Z_{i,t}}\leq 1$ for all $1\leq i \leq \sumcount$ and $t \in T$, where $Z_{i,t}$ is the $t$-coordinate of $Z_i$.
	Let $\sigma^2 := \sup_{t\in T} \sum_{i=1}^\sumcount\EE ((Z_{i,t})^2)$, $U := \sup_{t\in T} \sum_{i=1}^\sumcount (Z_{i,t})^2$ and\\ $Z:=\sup_{t\in T}  \left| \sum_{i=1}^\sumcount Z_{i,t} \right|$, then for any $\eta >0$
	\begin{align*}
		\PP  \left(  Z \geq \EE Z +  2\sqrt{(\sigma^2 + \EE U ) \eta} + 2 \eta \right)\leq e^{-\eta}.
	\end{align*}
\end{tthm}
The result is the following corollary.
\begin{cor}[Sub-exponential concentration inequality]\label{cor:secondconcineq}
	Let $\evcf$ be the eigenvalue counting function as defined in \eqref{eq:defevcf}, and let $\L^r_{m,1},...,\L^r_{m,\sumcount}$ be translates of $\Lambda^r_m \neq \emptyset$ such that $\min \{ \d_{\text{set}}(\L^r_{m,i},\L^r_{m,j}) \mid i \neq j \}>r$.
	%In the setting of Theorem \ref{theorem:expofcountablesup} with the additional restriction $N_{[]}(q, \cF, P) \leq V 2^{W q}$ for some $V,W\geq 1$ we have
	Then
	\begin{multline*}
		\PP\left( \sup_{x\in \mrcF^{-1} (\QQ)} \left| \frac{1}{\sumcount} \sum_{i=1}^\sumcount \frac{\admf (\Lambda^r_{m,i}, \o) (x)}{\abs{\Lambda_m}} - \EE \frac{\admf (\Lambda^r_m, \o)(x)}{\abs{\Lambda_m}} \right| \geq \kappa \right)\\
		\quad\leq  \exp \left( -\frac{1}{2}\left(\sqrt{\kappa +1} - 1\right)^2 \sumcount + \frac{K_2}{2}\sqrt{\sumcount} \right) \nonumber 
        % \\  		
        \quad\leq \exp \left( -\frac{1}{12}\kappa^2 \sumcount + \frac{K_2}{2}\sqrt{\sumcount} \right)
	\end{multline*}
	for all $0< \kappa \leq 1$ and $\sumcount \geq \left( \frac{K_2}{\kappa} \right)^2$, where
	\begin{align*}
		K_2=\left(\frac{120}{\log (3/2)}+\frac{4}{\log (2)}\right)\sum\limits_{q=0}^\infty 2^{-q} \sqrt{2+ + 2 q \log (2)}<\infty.
	\end{align*}
	Furthermore we have for $0< \kappa \leq 1$ 
	\begin{align*}
	\PP\left( \sup_{x\in \mrcF^{-1} (\QQ)} \left| \frac{1}{\sumcount} \sum_{i=1}^\sumcount \frac{\admf (\Lambda^r_{m,i}, \o) (x)}{\abs{\Lambda_m}} - \EE \frac{\admf (\Lambda^r_m, \o)(x)}{\abs{\Lambda_m}} \right| \geq \kappa \right)
		\leq  \exp \left( -\frac{\kappa^2}{24} \sumcount \right).
	\end{align*}
	for $\sumcount \geq \left(12\frac{K_2}{\kappa^2}\right)^2$.
\end{cor}

\begin{proof}
	We apply Theorem \ref{theorem:Massart} with $X_i:=\Pi_{\L^r_m} \circ \gamma_{z_i}$, $T=\tilde{\cF}$,
	\begin{align*}
		Z_i=\left(f(X_i)-\EE (f(X_i))\right)_{f \in \tilde{\cF}}=\left(\frac{\admf_{\L^r_m}(\Pi_{\L^r_m} \circ \gamma_{z_i})(x)}{\abs{\Lambda_m}}-\frac{\EE_{\PP_{\L^r_m}}\admf_{\L^r_m}(\cdot)(x)}{\abs{\Lambda_m}}\right)_{x \in \mrcF^{-1} (\QQ)}
\end{align*}		
	and $Z_{i,f}=f(X_i) - \EE f(X_i)$, since then
	\begin{align*}
		Z&=\sup_{f\in \cF}  \left| \sum_{i=1}^\sumcount \left(f(X_i) - \EE f(X_i)\right) \right|
	\end{align*}		
	is up to normalization the supremum we are interested in. We immediately see that $\EE (Z_{i,f})=0$ and because of $0\leq f \leq 1$ we also have $\abs{Z_{i,f}}\leq 1$. Since $X_1,...,X_\sumcount$ are independent and all $f\in \tilde{\cF}$ are measurable, we also know that $Z_{1,f},...,Z_{\sumcount,f}$ are independent for each $f$. Therefore, the vectors $Z_1,...,Z_\sumcount$ are independent. The requirements of the theorem are thus satisfied and we need to find bounds on $\sigma^2$ and $\EE U$. We have
	\begin{align*}
		\sigma^2 &=\sup_{f\in \tilde{\cF}} \sum_{i=1}^\sumcount \Var (f(X_i))= \sup_{f\in \tilde{\cF}} \sum_{i=1}^\sumcount \EE (f(X_i)^2)-(\EE (f(X_i)))^2 \\
		&\leq \sup_{f\in \tilde{\cF}} \sum_{i=1}^\sumcount \EE (f(X_i)^2) \leq \sup_{f\in \tilde{\cF}} \sum_{i=1}^\sumcount 1 = \sumcount
	\end{align*}
	and
	\begin{align*}
		U = \sup_{f\in \tilde{\cF}} \sum_{i=1}^\sumcount \left(f(X_i) - \EE f(X_i)\right)^2 \leq \sup_{f\in \tilde{\cF}} \sum_{i=1}^\sumcount 1 = \sumcount,
	\end{align*}
	because every $f$ is pointwise between $0$ and $1$. This implies $ \EE U \leq \sumcount$. Now we just need a bound on $\EE Z$. As stated after Definition \ref{def:orlicznorm}, the Orlicz norm generated by $\psi_{p,2}$ is identical to the one generated by $\Psi_p=\mathrm{e}^{(x^p)}-1\geq x^p$. Thus $\EE X=\norm[\Phi(x)=x]{X}\leq \norm[\Psi_1]{X}=\norm[\psi_{1,2}]{X}$ and by the same argument as in Corollary \ref{cor:firstconcineq} we have
	\begin{align*}
		\EE Z= \sqrt{\sumcount} \EE \left(  \sup_{x\in \mrcF^{-1} (\QQ)} \left| \frac{1}{\sqrt{\sumcount}} \sum_{i=1}^\sumcount \frac{\admf (\Lambda^r_{m,i}, \o) (x)}{\abs{\Lambda_m}} - \EE \frac{\admf (\Lambda^r_m, \o)(x)}{\abs{\Lambda_m}} \right|\right) \leq \sqrt{\sumcount} K_2.
	\end{align*}
	Applying Theorem \ref{theorem:Massart} then leads to
	\begin{align*}
		\PP\left( \sup_{f\in \tilde{\cF}} \frac{1}{\sumcount} \left| \sum_{i=1}^\sumcount (f(X_i) - \EE f(X_i)) \right| \geq \overbrace{\frac{ K_2  \sqrt{\sumcount} + 2 \sqrt{2\sumcount} \sqrt{x} + 2 x}{\sumcount}}^{=:\kappa (x,\sumcount)} \right) \leq e^{-x}.
	\end{align*}
	Now we need to find $x(\kappa, \sumcount)$, the inverse function of $x \to \kappa (x,\sumcount)$. Let $y:=\sqrt{x}$, then the equation we have to solve is
	\begin{align*}
		&K_2  \sqrt{\sumcount} + 2 \sqrt{2} \sqrt{\sumcount} y+ 2 y^2 = \kappa \sumcount \\
		\Leftrightarrow&y^2 + \sqrt{2}\sqrt{\sumcount} y + \frac{K_2 }{2}\sqrt{\sumcount} = \frac{\kappa}{2}\sumcount \\
		\Leftrightarrow&\left(y + \frac{1}{\sqrt{2}}\sqrt{\sumcount}\right)^2 - \frac{1}{2}\sumcount + \frac{K_2 }{2}\sqrt{\sumcount}=\frac{\kappa}{2}\sumcount\\
		\Leftrightarrow&\left(y + \frac{1}{\sqrt{2}}\sqrt{\sumcount}\right)^2 = \underbrace{\frac{\kappa }{2}\sumcount +\frac{1}{2}\sumcount - \frac{K_2 }{2}\sqrt{\sumcount}}_{\text{has to be }\geq 0}
	\end{align*}
	The equation can be solved if $\sumcount \geq \left(\frac{K_2}{\kappa + 1}\right)^2$ and the solution is
	\begin{align*}
		y=\sqrt{\frac{\kappa + 1}{2}\sumcount - \frac{K_2 }{2}\sqrt{\sumcount}}-\frac{1}{\sqrt{2}}\sqrt{\sumcount}
	\end{align*}
	This has to be positive, which is fulfilled if $\sumcount \geq \left( \frac{K_2 }{\kappa} \right)^2$, which also implies the previous condition.\\
	Under this condition we have
	\begin{align*}
		x&=y^2=\frac{\kappa + 2}{2}\sumcount - \frac{K_2 }{2}\sqrt{\sumcount} - \sqrt{2\sumcount}\sqrt{\underbrace{\frac{\kappa + 1}{2}\sumcount - \frac{K_2 }{2}\sqrt{\sumcount}}_{\leq \frac{\kappa + 1}{2}\sumcount}}\\
		&\geq \left( \frac{\kappa + 2}{2} - \sqrt{\kappa + 1} \right) \sumcount - \frac{K_2 }{2}\sqrt{\sumcount}=\frac{1}{2}\left( \sqrt{\kappa +1 } - 1 \right)^2 \sumcount- \frac{K_2 }{2}\sqrt{\sumcount}
	\end{align*}
	As a result
	\begin{align*}
		\PP&\left( \sup_{f\in \tilde{\cF}} \frac{1}{\sumcount} \left| \sum_{i=1}^\sumcount (f(X_i) - \EE f(X_i)) \right| \geq \kappa \right)\\
		&\leq \exp \left(-\left(\frac{\kappa + 2}{2} \sumcount - \frac{K_2 }{2}\sqrt{\sumcount} - \sqrt{2 \sumcount}\sqrt{\frac{\kappa + 1}{2}\sumcount - \frac{K_2 }{2}\sqrt{\sumcount}}\right)\right)\\
		&\leq  \exp \left( -\frac{1}{2}\left(\sqrt{\kappa +1} - 1\right)^2 \sumcount + \frac{K_2 }{2}\sqrt{\sumcount} \right).
	\end{align*}
	Next we use that $\left(\sqrt{\kappa +1}-1\right)^2 \geq \kappa^2 /6$ for all $0 \leq \kappa \leq 1$, which can be verified by using a substitution $\vartheta =\sqrt{\kappa+1}$. This gives us
	\begin{align*}
		\PP\left( \sup_{f\in \tilde{\cF}} \frac{1}{\sumcount} \left| \sum_{i=1}^\sumcount (f(X_i) - \EE f(X_i)) \right| \geq \kappa \right)\leq \exp \left( -\frac{1}{12}\kappa^2 \sumcount + \frac{K_2 }{2}\sqrt{\sumcount} \right).
	\end{align*}
	We get non-trivial results if
	\begin{align*}
		-\frac{1}{12}\kappa^2 \sumcount + \frac{K_2 }{2}\sqrt{\sumcount} \leq 0
	\end{align*}
	which is equivalent to
	\begin{align*}
		\sumcount \geq \left(6\frac{K_2 }{\kappa^2}\right)^2.
	\end{align*}
	Furthermore for
	\begin{align*}
		\sumcount \geq \left(12\frac{K_2 }{\kappa^2}\right)^2,
	\end{align*}
	we get
	\begin{align*}
		-\frac{1}{12}\kappa^2 \sumcount + \frac{K_2 }{2}\sqrt{\sumcount} \leq -\frac{1}{24}\kappa^2 \sumcount
	\end{align*}
	and therefore
	\begin{align*}
		\PP\left( \sup_{f\in \tilde{\cF}} \frac{1}{\sumcount} \left| \sum_{i=1}^\sumcount (f(X_i) - \EE f(X_i)) \right| \geq \kappa \right) &\leq \exp \left( -\frac{1}{12}\kappa^2 \sumcount + \frac{K_2 }{2}\sqrt{\sumcount} \right)
            %\\ 		&
        \leq \exp \left( -\frac{1}{24}\kappa^2 \sumcount\right)
	\end{align*}
	which is the last statement of the Corollary.
\end{proof}
These concentration inequalities can now be combined with the geometric approximations of Section \refeq{sec:geometric}.

\section{Proofs of the main results}\label{sec:proofsofthemaintheorems}
Here we prove Theorem \ref{thm:sqrtexpconcentrationinequality},
Remark \ref{cor:sqrtexpconcentrationinequality1+2dim},
and Theorem \ref{thm:subexpconcentrationinequality} spelled out in Section \ref{s:MR}.

\begin{proof}[Proof of Theorem \ref{thm:sqrtexpconcentrationinequality}]
	Applying Corollary \ref{cor:firstconcineq} for $n>2m$ with $s=\abs{T_{m,n}}=\lfloor n/m \rfloor^d$ and $\left(\L_{m,i}^r\right)_{1\leq i\leq s}=\left(\L_m^r +t\right)_{t \in T_{m,n}}$ leads to
	\begin{align*}
		\PP \left( A_{M,n,m,\kappa} \right) \geq 1- M \exp \left(- \frac{\sqrt{\sumcount} \kappa}{K_M}\right)
	\end{align*}
	for $M\geq 2$ and $K_M$ as in Corollary \ref{cor:firstconcineq}, where
	\begin{align*}
		A_{M,n,m,\kappa}=\left\{\o \in \O : \sup_{x\in \mrcF^{-1} (\QQ)} \left| \frac{1}{\abs{T_{m,n}}} \sum_{t \in T_{m,n}} \frac{\admf (\Lambda^r_m + t, \o) (x)}{\abs{\Lambda_m}} - \EE \frac{\admf (\Lambda^r_m, \o)(x)}{\abs{\Lambda_m}} \right| < \kappa\right\}.
	\end{align*}
	Now we want to extend the statement from the countable set of $x$ to the whole energy axis.
	By a Glivenko-Cantelli-type argument we obtain
	\begin{align*}
		\o \in A_{M,n,m,\kappa} \Rightarrow \norm[\infty]{\frac{1}{\abs{T_{m,n}}} \sum_{t \in T_{m,n}} \frac{\admf (\Lambda^r_m + t, \o) (x)}{\abs{\Lambda_m}} - \EE \frac{\admf (\Lambda^r_m, \o)(x)}{\abs{\Lambda_m}}} \leq \kappa ,
	\end{align*}
	mainly using monotonicity \ref{prop:admissiblemonotone} and some properties of $\mrcF$. For details see \cite[Lemma 6.18]{Diss}.\\
	In combination with \eqref{eq:completegeometricbound} we have
	\begin{align}\label{eq:nmkappabound}
		\norm[\infty]{\frac{\evcf  (\L_n, \o)}{\abs{\L_n}} -\admf^{*}}&\leq 32 d \frac{1}{n} + 104 \left(2^d-1\right) \frac{m}{n} +  2\left( 4d + 2r (2^d-1) + 36dr \right) \frac{1}{m}  + \kappa.
	\end{align}
	 for $\o \in A_{M,n,m,\kappa}$ and $n>4m$ and $m>2r+1$.\\
	Now we want to choose $m$ and $\kappa$ as functions of $n$ with the fastest convergence of the uniform error possible. Since $n>m$ the first term on the right side does not dominate the error, the relevant terms for the convergence are only $\frac{m}{n}$, $\frac{1}{m}$ and $\kappa$. The optimal choice for the first two of these is $m=\sqrt{n}$, since any other choice will either increase the first or the second. As $m$ needs to be a natural number we will choose $m(n)=\lfloor\sqrt{n}\rfloor$. From this follows that we can choose $\kappa (n)=\frac{1}{\lfloor\sqrt{n}\rfloor}$ without slowing the convergence.\\
	But now we need to check whether $\sqrt{\lfloor n/m(n) \rfloor^d} \kappa (n)$ still grows in $n$, otherwise there is no concentration. We have
	\begin{align*}
		\sqrt{\lfloor n/m(n) \rfloor^d} \kappa (n)\geq \sqrt{\lfloor n/\sqrt{n} \rfloor^d} \frac{1}{\lfloor\sqrt{n}\rfloor}=\lfloor\sqrt{n}\rfloor^{d/2-1},
	\end{align*}
	so our choice works for $d\geq 3$, but not for $d=1,2$. 
[We exclude these special cases here as we would need to choose different functions for $m(n)$ and $\kappa (n)$. The results for these cases are listed in Remark \ref{cor:sqrtexpconcentrationinequality1+2dim} and a proof can be found in \cite[Corollary 7.3]{Diss}.]
 We required $m>2r+1$ before, so $n>m$ is true as long as  $n>(2r+1)^k$. We also need to ensure $n>4m$, but since
	\begin{align*}
		\frac{n}{m}=\frac{n}{\lfloor\sqrt[k]{n}\rfloor}\geq n^{1-\frac{1}{k}}\geq \sqrt{n}
	\end{align*}
	this is achieved by $n>16$.
	The result we arrive at with this choice is
	\begin{align*}
		\norm[\infty]{\frac{\evcf  (\L_n, \o)}{\abs{\L_n}} -\admf^{*}} &\leq 32 d \frac{1}{n} +104 \left(2^d-1\right) \frac{\lfloor\sqrt{n}\rfloor}{n} +  \left( 8d + 4r (2^d-1) + 72dr +1 \right) \frac{1}{\lfloor\sqrt{n}\rfloor}\\
		&\leq 32 d \frac{1}{n} +104 \left(2^d-1\right) \frac{1}{\sqrt{n}} +  \left( 8d + 4r (2^d-1) + 72dr +1 \right) \frac{1}{\sqrt{n}-1}
	\end{align*}
	as claimed in \eqref{eq:evcfbound} The identification of $\admf^{*}$ with $N$ was already established in \cite[Theorem 7.2]{SSV17}.
\end{proof}

\begin{proof}[Proof of Remark \ref{cor:sqrtexpconcentrationinequality1+2dim}]
	To check what we have to change for $d=1$ and $d=2$ we start out with a general case of $m(n)\sim n^j$ for some $j<1/2 \in \RR$ and $\kappa (n)=1/m(n)$. The convergence of the uniform bound will then be dominated by $n^{-j}$, since $\frac{m}{n}\sim n^{-(1-j)}$. We have to ensure that
	\begin{align*}
		\sqrt{\lfloor n/m(n) \rfloor^d} \kappa (n) \sim n^{\frac{d(1-j)}{2}-j}
	\end{align*}
	grows in $n$, i.e.{}
	\begin{align*}
		j<\frac{d}{2(1+d/2)}=\begin{cases}
			\frac{1}{3} & \text{ for } d=1\\
			\frac{1}{2} & \text{ for } d=2
		\end{cases}.
	\end{align*}
	Thus, we choose
	\begin{align*}
		d=1: \ &m(n)=\lfloor\sqrt[4]{n}\rfloor, \ \kappa (n)=\frac{1}{\lfloor\sqrt[4]{n}\rfloor} \\
		d=2: \ &m(n)=\lfloor\sqrt[3]{n}\rfloor, \ \kappa (n)=\frac{1}{\lfloor\sqrt[3]{n}\rfloor}.
	\end{align*}
	and arrive at
	\begin{align*}
		\norm[\infty]{\frac{\evcf  (\L_n, \o)}{\abs{\L_n}} -\admf} &\leq 32 d \frac{1}{n} +104 \left(2^d-1\right) \frac{\lfloor\sqrt[k]{n}\rfloor}{n} +  \left( 8d + 4r (2^d-1) + 72dr +1 \right) \frac{1}{\lfloor\sqrt[k]{n}\rfloor}\\
		&\leq 32 d \frac{1}{n} +104 \left(2^d-1\right) \frac{1}{n^{1-1/k}} +  \left( 8d + 4r (2^d-1) + 72dr +1 \right) \frac{1}{\sqrt[k]{n}-1}
	\end{align*}
	where
	\begin{align*}
		k=\begin{cases}
			4 &  \text{ for } d=1\\
			3 &  \text{ for } d=2
		\end{cases}.
	\end{align*}		
\end{proof}
\begin{proof}[Proof of Theorem \ref{thm:subexpconcentrationinequality}]
	The bound in \eqref{eq:evcfbound:k} follows just as in the proof of Theorem \ref{thm:sqrtexpconcentrationinequality}, but using the concentration inequality of Corollary \ref{cor:secondconcineq} instead of Corollary \ref{cor:firstconcineq}. Just as in that proof the best choice for $m(n)$ arising from the uniform bound would be $m(n)=\lfloor \sqrt{n}\rfloor$, but we also need to ensure that
	\begin{itemize}
		\item $\lfloor n/m(n) \rfloor^d /m(n)^2$ is growing in $n$ (which would be satisfied by the same $m(n)$ calculated for Theorem \ref{thm:sqrtexpconcentrationinequality}),
		\item  $\lfloor n/m \rfloor^d /m(n)^4 \geq \left(12 K_2\right)^2$ for the formulas from Corollary \ref{cor:secondconcineq} to be applicable.
	\end{itemize}
	By setting $m(n)\sim n^{1/k}$ for some yet to be determined $k$ we get
	\begin{align*}
		\frac{\lfloor n/m \rfloor^d}{m(n)^4} \sim n^{d(1-1/k)-4/k},
	\end{align*}
	and thus need to ensure that $d(1-1/k)-4/k>0$ which is equivalent to $k>\frac{4+d}{d}$. Our favoured case $k=2$ satisfies this condition only for $d \geq 5$. We have
	\begin{align*}
		\frac{\lfloor n/m \rfloor^d}{m(n)^4}\geq \frac{\left(n^{1-1/k}-1\right)^d}{n^{4/k}}
	\end{align*}
	and ensure the applicability if
	\begin{align*}
		n^{1-1/k}-1 \geq n^{4/(dk)} \left(12 K_2\right)^{2/d}
	\end{align*}
	or equivalently
	\begin{align*}
		n^{1-1/k-4/(dk)}-n^{-4/(dk)}\geq \left(12 K_2\right)^{2/d}.
	\end{align*}
	Since
	\begin{align*}
		n^{1-1/k-4/(dk)}-n^{-4/(dk)} \geq n^{1-1/k-4/(dk)}-1
	\end{align*}
	this is guaranteed if
	\begin{align*}
		n^{1-1/k-4/(dk)}-1\geq \left(12 K_2\right)^{2/d}
	\end{align*}
	or equivalently
	\begin{align*}
		n >  \left(\left(12 K_2\right)^{2/d}+1\right)^{\frac{d}{d-\frac{(d+4)}{k}}}
	\end{align*}
	as claimed.
\end{proof}
\section{Extensions to other settings}
\label{s:EoS}
The setting discussed in this paper only covered eigenvalue counting functions for the discrete Schrödinger operator on the lattice $\ZZ^d$ with restrictions to cubes. With a few changes to the geometric approximations used it is also possible to choose any monotiling F{\o}lner sequence instead of cubes. These changes can be found in \cite[Section 5.2]{Diss}, and the results corresponding to our main results can be found in \cite[Section 7.2]{Diss}. Instead of $\ZZ^d$ it is also possible to consider Cayley graphs of finitely generated amenable groups and eigenvalue counting functions on a nested F{\o}lner sequences, which is covered in \cite[Section 8.1]{Diss}. Here we need to resort to $\eps$-quasi tilings, leading to slightly less exact quantifications.\\
In the proofs presented here we often used the fact that the Schrödinger operator has a finite hopping range, but this is assumption can be weakened. For instance, we can also derive concentration inequalities for eigenvalue counting functions of the Laplace operator on long-range percolation graphs. This is done in \cite[Section 8.2]{Diss} for $\ZZ^d$ and cubes.

\subsection*{Acknowledgements}
Work on this paper was partially supported by the \emph{Deutsche Forschungsgemeinschaft}  through grant no. VE 253:9-1 
\emph{Random Schrödinger operators with non-linear influence of randomness} 
and 
through the GRK 2131 \emph{High-dimensional Phenomena in Probahility - Fluctuations and Discontinuity}

\bibliographystyle{alphaurl}

\end{document}